\documentclass[11pt]{article}

\usepackage[english]{babel}
\usepackage[utf8]{inputenc}

\usepackage[left=1.1 in, right= 1.1 in, top = 1 in, bottom = 1 in]{geometry}

\makeatletter
\g@addto@macro\normalsize{%
  \setlength\abovedisplayskip{7pt}
  \setlength\belowdisplayskip{7pt}
  \setlength\abovedisplayshortskip{7pt}
  \setlength\belowdisplayshortskip{7pt}
}
\makeatother

\usepackage{tocloft}

\usepackage{amsfonts}
\usepackage{mathrsfs}
\usepackage{bbm}
\usepackage{latexsym}
\usepackage{math dots}
\usepackage{amssymb}
\usepackage{mathtools}
\usepackage{relsize}
\usepackage{lipsum}

\usepackage{todonotes}
\usepackage{enumitem}
\setlist{nolistsep} 	
\usepackage{amsthm}

\usepackage{xcolor}
\definecolor{Color1}{rgb}{0.0, 0.42, 0.47}
\definecolor{Color2}{rgb}{0.78, 0.11, 0.0}

\usepackage{titlesec}
\titleformat{\section}
  {\large\center\bfseries}
  {\thesection.}{.7em}{}
\titlespacing*{\section}{0pt}{3.5ex plus 0ex minus 0ex}{1.5ex plus 0ex}
\titleformat{\subsection}
  {\bfseries}
  {\thesubsection.}{.7em}{}
\titlespacing*{\subsection}{0pt}{3.5ex plus 0ex minus 0ex}{1.5ex plus 0ex}
\titleformat{\subsubsection}
  {\center\bfseries}
  {\thesubsubsection.}{.7em}{}
\titlespacing*{\subsubsection}{0pt}{3.5ex plus 0ex minus 0ex}{1.5ex plus 0ex}

\addto\captionsenglish{}

\makeatletter
\renewenvironment{abstract}{
\begin{center}
{\bfseries \large\abstractname\vspace{\z@}}
\end{center}
\quotation
}
\makeatother

\makeatletter
\@namedef{subjclassname@2020}{
  \textup{2020} Mathematics Subject Classification}
\makeatother

\usepackage[linktocpage=true]{hyperref}
\usepackage[capitalize]{cleveref}
\hypersetup{citecolor = Color2,colorlinks,
			linkcolor = black,
			urlcolor = Color2}

\newtheoremstyle{plain}{3mm}{3mm}{\slshape}{}{\bfseries}{.}{.5em}{}
\newtheoremstyle{definition}{2mm}{2mm}{}{}{\bfseries}{.}{.5em}{}
\theoremstyle{plain}
	
\newtheorem{Theorem}{Theorem}

\newtheorem{Lemma}[Theorem]{Lemma}
\newtheorem{lemma}[Theorem]{Lemma}
\newtheorem{Proposition}[Theorem]{Proposition}

\newtheorem{Question}[Theorem]{Question}
\newtheorem{Problem}[Theorem]{Problem}

\theoremstyle{definition}
\newtheorem{Definition}[Theorem]{Definition}
\newtheorem{Remark}[Theorem]{Remark}

\newtheorem*{Theorem*}{Theorem}

\theoremstyle{plain} 
\newcounter{MainTheoremCounter}

\theoremstyle{plain}
\newtheorem*{namedthm}{\namedthmname}
\newcounter{namedthm}
\makeatletter
	\newenvironment{named}[2]
	{\def\namedthmname{#1}
	\refstepcounter{namedthm}
	\namedthm[#2]\def\@currentlabel{#1}}
	{\endnamedthm}
\makeatother

\usepackage{chngcntr}
\counterwithin{Theorem}{section}

\numberwithin{equation}{section}

\allowdisplaybreaks

\newcommand{\Mobius}{M\"{o}bius}

\newcommand{\Lemanczyk}{Lema\'{n}czyk}

\newcommand{\N}{\mathbb{N}}
\newcommand{\Z}{\mathbb{Z}}
\newcommand{\R}{\mathbb{R}}
\newcommand{\C}{\mathbb{C}}

\newcommand{\T}{\mathbb{T}}

\renewcommand{\epsilon}{\varepsilon}
\renewcommand{\leq}{\leqslant}
\renewcommand{\geq}{\geqslant}
\renewcommand{\setminus}{\backslash}

\renewcommand{\P}{\mathbb{P}}

\renewcommand{\d}{~\mathrm{d}}

\usepackage[normalem]{ulem}

\author{By~~{\scshape Andreas Koutsogiannis}, {\scshape Anh N.\ Le},~~{\scshape Joel Moreira},\\ ~~{\scshape Ronnie Pavlov} and~~{\scshape Florian~K.~Richter}}
\date{}

\title{\bfseries Interpolation sets for dynamical systems}

\begin{document}

\maketitle
\begin{abstract} 
\noindent Originating in harmonic analysis, interpolation sets were first studied in dynamics by Glasner and Weiss in the 1980s.
A set $S\subset\N$ 
is an interpolation set for a class of topological dynamical systems $\mathcal{C}$ if any bounded sequence on $S$ can be 
extended to a sequence that arises from 
a system in $\mathcal{C}$.
In this paper, we provide 
combinatorial characterizations of interpolation sets for:
\begin{itemize}
\item (totally) minimal systems;
\item topologically (weak) mixing systems;
\item strictly ergodic systems; and
\item zero entropy systems.
\end{itemize}
\noindent 
Additionally, we prove some results on a slightly different notion, called weak interpolation sets, for several classes of systems.
We also answer a question of Host, Kra, and Maass concerning the connection between sets of pointwise recurrence for distal systems and $IP$-sets.

\end{abstract}

\renewcommand{\thefootnote}{\fnsymbol{footnote}} 
\renewcommand{\thefootnote}{\arabic{footnote}}

\tableofcontents

\thispagestyle{empty}

\section{Introduction}

Several studies have explored the notion of interpolation/interpolating sets for classes of sequences arising from dynamical systems \cite{Glasner-Tsankov-Weiss-Zucker-bernoulli-disjointness, Glasner-Weiss-Interpolation, Glasner_Weiss95,  Le_interpolation_nil, Le_sublac, Weiss-SingleOrbitDynamics}. 
The first such definition is due to Glasner \cite{Glasner_divisibility}: letting $\mathcal{A}$ be a subset of $\ell^{\infty}(\N)$ -- the space of all bounded complex-valued functions on the set of positive integers $\N=\{1,2,3,\ldots\}$ --
a set $S \subset \N$ is called an 
\emph{interpolation set for $\mathcal{A}$} if every bounded function $f\colon S\to\C$ can be extended to a function in $\mathcal{A}$. The collection of all interpolation sets for $\mathcal{A}$ is denoted by $\mathcal{I}_{\mathcal{A}}$. 


A \emph{topological dynamical system} is a pair $(X, T)$ where $X$ is compact metric space and $T:X \to X$ is a continuous map. The system $(X, T)$ is \emph{minimal} if for every $x \in X$, the (forward) orbit $o_T(x):=\{T^n x: n \in \N_0\}$ is dense in $X$, where $\N_0:= \N \cup \{0\}$. 
Let $\mathcal{M} \subset\ell^\infty(\N)$ denote the class of bounded sequences $x$ for which the closure of the orbit of $x$ under the left shift map $\sigma$ forms a minimal dynamical system and let $\mathcal{U}$ be the smallest norm-closed, translation-invariant subalgebra of $\ell^{\infty}(\N)$ containing $\mathcal{M}$. 
An old question of Furstenberg \cite{Furstenberg-disjointness-in-ergodic-theory} asked whether $\mathcal{U} = \ell^{\infty}(\N)$.
Glasner and Weiss \cite{Glasner-Weiss-Interpolation} provided a negative answer to this question (for $\mathbb{Z}$-subshifts rather than $\mathbb{N}$-subshifts) through an investigation into the interpolation sets for $\mathcal{U}$. If $\mathcal{U} = \ell^{\infty}(\N)$, every subset of $\N$ would be an interpolation set for $\mathcal{U}$. However, Glasner and Weiss \cite{Glasner-Weiss-Interpolation} showed that an interpolation set for $\mathcal{U}$ cannot be piecewise syndetic (see \cref{sec_background} for definition).
Recently, Glasner, Tsankov, Weiss, and Zucker \cite{Glasner-Tsankov-Weiss-Zucker-bernoulli-disjointness} generalized this result from $\Z$  to arbitrary infinite discrete groups.

The work of Glasner and Weiss mentioned above was motivated partly by a classical notion in harmonic analysis called Sidon sets.
Letting $\mathcal{F}\subset \ell^\infty(\N)$ be all Fourier transforms of measures on the unit circle, a set $S \subset \N$ is \emph{Sidon} 
if it is an interpolation set for $\mathcal{F}$.
Sidon \cite{Sidon-fourier-1, Sidon-fourier-2} himself showed that every lacunary set is Sidon.\footnote{A set $\{a_n: n \in \N\} \subset \N$ with $a_1 < a_2 <\ldots$ is \emph{lacunary} if $\inf_{n \in \N} a_{n+1}/a_n > 1$.} A later result of Drury \cite{Drury_Sidonsets} showed that the union of two Sidon sets is Sidon.

An important special case of Sidon sets arises when one considers 
$\mathcal{F}_0$ -- the set of Fourier transforms of discrete measures on the unit circle.  
The interpolation sets for $\mathcal{F}_0$ are called \emph{$I_0$-sets}.
$I_0$-sets were studied extensively starting in the 1960s \cite{Hartman_1961, Hartman_Ryll-Nardzewski_1964, Kunen_Rudin_1999, Ramsey77, Ryll-Nardzewski_1964}. Strzelecki \cite{Strzelecki_1963} showed that lacunary sets are $I_0$-sets, extending Sidon's results \cite{Sidon-fourier-1, Sidon-fourier-2}. 
Ryll-Nardzewski proved that the union of an $I_0$-set and a finite set is an $I_0$-set \cite{Ryll-Nardzewski_1964}. 
Using Hartman-Ryll-Nardzewski characterization \cite{Hartman_Ryll-Nardzewski_1964}, some unions of lacunary sets such as $\{2^n: n \in \N\} \cup \{2^n + 1: n \in \N\}$ are $I_0$-sets. However, in contrast to Sidon sets, the union of two $I_0$-sets is not necessarily $I_0$, for example, $\{2^n: n \in \N\} \cup \{2^n + n: n \in \N\}$ \cite{Ryll-Nardzewski_1964}. It is not true either that every $I_0$-set is a finite union of lacunary sets as seen in the following example by Grow \cite{Grow_1987} and M\'ela \cite{Mela_1969}: $\{3^{n^2} + 3^j: n \geq 1, (n-1)^2 \leq j \leq n^2\}$.

The building blocks of previously known examples of $I_0$-sets are lacunary sets, a feature that makes them extremely sparse. Therefore, it is natural to ask whether there exists an $I_0$-set that has subexponential\footnote{$\{a_n: n \in \N\} \subset \N$ with $a_1 < a_2 < \ldots$ is \emph{subexponential} if $\lim_{n \to \infty} (\log a_n)/n = 0$.} or even polynomial growth. Using a dynamical method, Le \cite{Le_interpolation_nil} gave a negative answer to this question. This result was extended to interpolation sets for nilsequences in \cite{Briet_Green_2021, Le_sublac}.
Another connection between interpolation sets and dynamics was discovered by Griesmer \cite{Griesmer_equivalentKatznelson}, who established
a link between $I_0$-sets and Katznelson's question on the equivalence of Bohr recurrence and topological recurrence.

Interpolation sets also connect to questions in number theory. 
Sarnak's \Mobius{} disjointness conjecture \cite{Sarnak12}, for example, lies at the intersection of multiplicative number theory and dynamics and has received a lot of attention in the last decade (see \cite{Ferenczi-Kulaga-Przymus-Lemanczyk-sarnak-survey, Kulaga-Przymus-Lemanczyk-survey} for surveys on recent developments).
Pavlov showed in \cite{Pavlov-polynomialSarnak} that every set of Banach density $0$ is a weak interpolation set (\cref{def_interpolation_weak}) for the class of all finite-valued sequences whose shift closure forms a zero-entropy minimal dynamical system. This implies that Sarnak's conjecture fails if one replaces the sequence of integers with any sequence of zero Banach density. In particular, this disproved a polynomial analogue of Sarnak's conjecture posed in \cite{Eisner-polynomialSarnak}, which has previously been addressed independently by Kanigowski, Lema\'nczyk, and Radziwi\l\l \ \cite{Lemanczyk-Kanigowski-Radziwill-primenumbertheoremskewproduct} and Lian, Shi \cite{Lian-Shi-counterexample}.

Given a rich but scattered literature on interpolation sets in dynamics, the primary goal of this work is a systematic study of interpolation sets for dynamical systems.
Specifically, we characterize the interpolation sets for sequences arising from following classes of topological dynamical systems: systems with various mixing-type properties, minimal and totally minimal systems, uniquely and strictly ergodic systems, and systems with restricted entropy. 
We begin our exploration with precise definitions.

\begin{Definition}[Strong interpolation]\label{def_interpolation_regular}
    Let ${\mathcal C}$ be a class of 
    topological dynamical systems. A set $S\subset\N$ is an \emph{interpolation set for ${\mathcal C}$} if for every bounded function $f:S\to\C$ there exists a system $(X,T)\in{\mathcal C}$, a transitive\footnote{A point $x \in X$ is \emph{transitive} if its orbit, $o_T(x) = \{T^n x: n \in \N_0\}$, is dense in $X$.} point $x\in X$ and $F\in C(X)$ such that $F(T^nx)=f(n)$ for every $n\in S$.
\end{Definition}

In other words, $S$ is an interpolation set for $\mathcal{C}$ if and only if $S\in\mathcal{I}_{\mathcal{F}}$ for 
$\mathcal{F}=\{n\mapsto F(T^nx): (X,T)\in\mathcal{C},~F\in C(X),~x\in X~\text{transitive}\}$. To differentiate with a weaker notion of interpolation sets we are about to introduce, occasionally we refer to interpolation sets as \emph{strong interpolation sets}. 





In some situations, it will be useful to have a version of interpolation sets which works for functions from the set $S$ to a finite set, rather than to $\mathbb{C}$. 
The size of the finite set in some sense quantifies the ``amount of independence'' to be exhibited. 
A similar idea was already considered by Glasner and Weiss \cite[Section 3]{Glasner_Weiss95} (see also \cite[Theorem 8.1]{Weiss-SingleOrbitDynamics}) where they 
proved that if $X$ is a 
$\{0, 1\}$-subshift of positive entropy, then there exists $S \subset \N$ of positive density such that $X|_S = \{0, 1\}^\N|_S$, where we write $x|_S$ to denote the function obtained by restricting the domain of $x\in \{0,1\}^\N$ to $S\subset\N$, and denote by $Y|_S$ the set $\{x|_S:x\in Y\}$ whenever $Y\subset\{0,1\}^\N$. 
In that spirit, we make the following definition.

\begin{Definition}[Weak interpolation]\label{def_interpolation_weak}
    Let ${\mathcal C}$ be a class of 
    topological dynamical systems and let $k\in\N$. 
    A set $S\subset\N$ is a \emph{weak interpolation set of order $k$ for ${\mathcal C}$} if for every bounded $f:S\to\{0,\dots,k-1\}$ there exists a system $(X,T)\in{\mathcal C}$, a transitive point $x\in X$ and $F\in C(X)$ such that $F(T^nx)=f(n)$ for every $n\in S$. We say that $S$ is a \emph{weak interpolation set of all orders} for $\mathcal{C}$ if it is a weak interpolation set of order $k$ for $\mathcal{C}$ for all $k \in \mathbb{N}$. 
\end{Definition}


There are some obvious relations between these notions of interpolation sets. 
It is immediate from the definition that whenever $k' \leq k$, every weak interpolation set of order $k$ is a weak interpolation set of order $k'$. 
Similarly, a strong interpolation set is a weak interpolation set of all orders.



For many classes of systems, the differences between weak interpolation sets of different orders are genuine. 
For example, since there exists a minimal and uniquely ergodic $\{0,1\}$-subshift of positive entropy \cite{Hahn-Katznelson1967}, Weiss's result \cite[Theorem 8.1]{Weiss-SingleOrbitDynamics} implies that there is a positive density set $S$ which is weak interpolation of order $2$ for the class of uniquely ergodic systems.  
On the other hand, in \cref{prop:uniquely_ergodic_one} below, we prove that every set of positive Banach density (such as $S$) cannot be a weak interpolation set of all orders for uniquely ergodic systems.


In general, it is also false that weak interpolation of all orders implies strong interpolation. Our \cref{prop:weak-finite-entropy} shows that for the class of finite entropy systems, all sets (including $\mathbb{N}$ itself) are weak interpolation of all orders, while \cref{prop:interpolation_finite_entropy} shows that only sets of zero Banach density are strong interpolation. 

That being said, for certain classes of systems, the concepts ``weak interpolation of order $k$'' for a specific value of $k$, ``weak interpolation of all orders,'' and ``strong interpolation'' are all equivalent. 
These classes include compact abelian group rotations \cite[Theorem 1]{Hartman_Ryll-Nardzewski_1964}, nilsystems \cite[Theorem 2.2]{Le_sublac}, and in general any class $\mathcal{C}$ that is closed under taking subsystems and countable Cartesian products (see \cref{sec:equivalence}).
Additionally, \cref{thm:interpolation_minimal_systems_new} below implies that the class of minimal systems also has this property (as a byproduct of a full classification of interpolation sets, of any kind, for the class of minimal systems).

The notions of interpolation sets we discussed above also connect to the notions of ``null'' and ``tame'' systems, 
which are important classes with several useful equivalent definitions. 
\emph{Null systems} are those with zero topological sequential entropy for every subsequence, and \emph{tame systems} are those whose Ellis enveloping semigroup has cardinality at most $2^{\aleph_0}$. 
For minimal systems, null is strictly stronger than tame, which is strictly stronger than uniquely ergodic with (measurable) discrete spectrum (see \cite{HuangTame}).
The connection to interpolation sets arises from alternative definitions of these notions in terms of so-called \emph{independence sets} (see \cite{KerrLi}). 
For instance, in \cite{Kerr_Li-indepedence-topological} it is shown that a system is non-tame if and only if there exist disjoint compact sets $K_0, K_1$ and an infinite set $S \subset \mathbb{N}$ so that for every $f: S \rightarrow \{0,1\}$, there exists $x \in X$ with $T^s x \in K_{f(s)}$ for all $s \in S$. 
A system is non-null if and only if there exist disjoint compact sets $K_0, K_1$ and arbitrarily large finite sets $S$ with the above properties. 
These conditions are not equivalent to ours; for instance, for any non-tame system $(X,T)$, there is an infinite weak interpolation set $S$ of order $2$ for the singleton class $\mathcal{C} = \{(X, T)\}$
(by using Urysohn's Lemma with the sets $K_0, K_1$), but it's not clear whether the converse is true (since in our definition, every function from $S$ to $\{0,1\}$ could correspond to a different $F \in C(X)$).

\subsection{Totally transitive, weak mixing, and strong mixing}


Three central notions of mixing in topological dynamical systems are (strong) mixing, weak mixing and total transitivity (see \cref{sec_background} for the definitions). 
The following implications are well known: mixing $\Rightarrow$ weak mixing $\Rightarrow$ totally transitive. 
Therefore, by definition, 
\begin{eqnarray*}
        &\text{interpolation for mixing systems}&
        \\
        &\Downarrow&
        \\
        &\text{interpolation for weak mixing systems}&
        \\
        &\Downarrow&
        \\
        &\text{interpolation for totally transitive systems}.& 
\end{eqnarray*}
Our first result asserts that the converse directions of both implications above also hold, and provides an easy-to-check combinatorial description of interpolation sets for these classes. 

\begin{Theorem}\label{thm:total_transitive_weak_mixing_mixing}
Let $S \subset \N$. The following are equivalent:
\begin{enumerate}
    \item 
    $S$ is a weak interpolation set of all orders for totally transitive systems.
    
    \item 
    $S$ is an interpolation set for totally transitive systems.
    
    \item 
    $S$ is an interpolation set for weak mixing systems.
    
    \item 
    $S$ is an interpolation set for mixing systems.

    \item 
    $S$ is not syndetic, i.e. $S$ has arbitrarily large gaps (see \cref{sec_background} for definition). 

\end{enumerate}
\end{Theorem}

We do not know whether weak interpolation of order $2$ for totally transitive systems is also equivalent to the condition that $S$ is not syndetic; see \cref{Q2synd}.

\begin{Remark}
    It is natural to wonder why we did not consider interpolation sets for the class of transitive systems. 
    This is because by our definition, every subset of $\N$ is such an interpolation set. Indeed, let $z = (z(n))_{n \in \N_0}$ be an arbitrary bounded sequence and let $X$ be the closure of the orbit of $z$ under the left shift $\sigma$. 
    By definition, the system $(X, \sigma)$ is transitive with $z$ as a transitive point. In addition, $F(\sigma^{n} z) = z(n)$ for all $n \in \N$ for the continuous function $F: X \to \C$ defined by $F((x(n))_{n \in \N_0}) = x(0)$. 
\end{Remark}

\subsection{Minimal and totally minimal systems}

Recall that $\mathcal{M}$ denotes the class of bounded sequences whose shift closure forms a minimal dynamical system and $\mathcal{U}$ is the closed, shift invariant subalgebra of $\ell^{\infty}(\N)$ generated by $\mathcal{M}$. 
In \cite[Theorem 1]{Glasner-Weiss-Interpolation}, Glasner and Weiss showed that $S$ is an interpolation set for $\mathcal{U}$ if and only if $S$ is not piecewise syndetic.\footnote{See \cref{sec_background} for definition; in \cite{Glasner-Weiss-Interpolation}, non-piecewise syndetic sets are called \emph{small sets}.} (See also \cite[Theorem 10.2]{Glasner-Tsankov-Weiss-Zucker-bernoulli-disjointness}.)
Since $\mathcal{M} \subset \mathcal{U}$, it follows that an interpolation set for $\mathcal{M}$ cannot be piecewise syndetic. 
Conversely, in the same paper \cite{Glasner-Weiss-Interpolation}, they also proved that every non-piecewise-syndetic set is a weak interpolation set of order 2 for $\mathcal{M}$. 
These results raise the natural question: is every non-piecewise-syndetic set a (strong) interpolation set for $\mathcal{M}$? 
 Our next theorem confirms that this is indeed the case.  
In addition, we show that every non-piecewise syndetic set is  an interpolation set for a smaller class, namely, totally minimal systems. (The system $(X,T)$ is \emph{totally minimal} if $(X,T^n)$ is minimal for all $n\in\N.$)

\begin{Theorem}\label{thm:interpolation_minimal_systems_new}
Let $S \subset \N$. The following are equivalent:
\begin{enumerate}
    \item 
    $S$ is a weak interpolation set of order $2$ for minimal systems. 
    
    \item 
    $S$ is a weak interpolation set of all orders for minimal systems.
    
    \item 
    $S$ is an interpolation set for minimal systems.
    
    \item 
    $S$ is an interpolation set for totally minimal systems.

    \item 
    $S$ is not piecewise syndetic.
\end{enumerate}
\end{Theorem}

\subsection{Systems of bounded entropy}

Topological entropy is a measure of the complexity of a dynamical system (see \cref{def:entropy}). 
It turns out that classifying interpolation sets for the class of systems whose entropy is bounded by a fixed constant is simple: 
such sets must have zero Banach density.

\begin{Proposition}\label{prop:interpolation_finite_entropy}
Let $S \subset \N$ and $M\geq0$. The following are equivalent:
\begin{enumerate}
    \item $S$ is an interpolation set for systems of finite topological entropy.

    \item $S$ is an interpolation for systems of topological entropy $\leq M$.

    \item $S$ is an interpolation set for systems of zero topological entropy.
    
    \item $S$ has zero Banach density, i.e., $d^*(S) = 0$ (see \cref{sec_background} for definition).
\end{enumerate}
\end{Proposition}

As a consequence of the fact that the full shift on a finite alphabet has finite entropy, we obtain the following observation regarding weak interpolation sets.

\begin{Proposition}\label{prop:weak-finite-entropy}
    Every subset $S \subset \N$ is a weak interpolation set of all orders for systems of finite entropy.
\end{Proposition}

While both of the above results are relatively simple, it turns out that a more interesting phenomenon emerges when considering weak interpolation sets of a fixed order $k$.
To state the next theorem, we define the entropy function $H: [0, 1] \to [0, \log 2]$ as
\[
     H(\delta) = \begin{cases} 
        - \delta \log \delta - (1 - \delta) \log (1 - \delta), \text{ if } \delta \in (0, 1)\\
        0, \text{ if } \delta = 0 \text{ or } 1.
    \end{cases}   
\]

\begin{Theorem}\label{thm_entropydensityclassification_intro_new}
Let $\delta\in[0,1/2]$. For any $k \in \N$ and $S \subset \N,$ we have the following: 
\begin{enumerate}
\item \label{item:entropy_less_delta}
If $d^*(S)\leq \delta$, then $S$ is a weak interpolation set of order $k$ for systems of entropy $\leq H(\delta) + \delta \log(k-1)$.

\item \label{item:entropy_greater_delta} If $d^*(S) > \delta$, then $S$ is not a weak interpolation set of order $k$ for the class of systems of entropy $\leq \delta \log k$.

\end{enumerate}
Moreover, (\ref{item:entropy_greater_delta}) is sharp in the following sense:
\begin{enumerate}
\item [3.] \label{item:entropy_equal} There exists $S \subset \N$ with $d^*(S) = \delta$ such that $S$ is a weak interpolation set of order $k$ for the class of systems of entropy $\leq \delta \log k$.
\end{enumerate}
\end{Theorem}
We do not know whether (\ref{item:entropy_less_delta}) is sharp, i.e. whether the bound $H(\delta) + \delta \log(k-1)$ could be improved, and leave this as a question.

\subsection{Uniquely ergodic, strictly ergodic, and zero entropy}

A topological system $(X,T)$ is \emph{uniquely ergodic} if there is only one Borel measure $\mu$ on $X$ that is preserved by $T$ (i.e., $\mu(A) = \mu(T^{-1}A)$ for every Borel set $A\subset X$), and is \emph{strictly ergodic} if it is uniquely ergodic and minimal. 
The next theorem characterizes interpolation sets for uniquely ergodic systems, strictly ergodic systems, and strictly ergodic systems of zero entropy. (We note that as in \cref{prop:interpolation_finite_entropy}, the characterization is in terms of zero Banach density; we elected not to combine the theorems due to the different proof techniques and for readability.)

\begin{Theorem}\label{prop:uniquely_ergodic_one}
Let $S \subset \N$. The following are equivalent:
\begin{enumerate}
    \item 
    $S$ is a weak interpolation set of all orders for uniquely ergodic systems.
    
    \item 
    $S$ is an interpolation set for uniquely ergodic systems.

    \item 
    $S$ is an interpolation set for strictly ergodic systems.

    \item $S$ is an interpolation set for strictly ergodic systems of zero entropy.

    \item 
    $d^*(S) = 0$.
\end{enumerate}
\end{Theorem}

As we have mentioned before, by combining Hahn-Katznelson's \cite{Hahn-Katznelson1967} and Weiss's results \cite[Theorem 8.1]{Weiss-SingleOrbitDynamics}, there is a set $S$ of positive density such that $S$ is weak interpolation of order $2$ for the class of strictly ergodic systems. This fact and \cref{prop:uniquely_ergodic_one} imply two things: First, for the classes of uniquely ergodic systems and strictly ergodic systems, weak interpolation of order $2$ and weak interpolation of all orders are not equivalent. 
Second, it leads to the question: Can a weak interpolation set of order $2$ for uniquely ergodic systems be syndetic? At the moment, we do not know the answer to this question (see \cref{Q2synd-2}). However, the answer is negative if we impose an additional condition on the growth of the corresponding complexity function:
\begin{Proposition}\label{prop:weak_interpolation_uniquely_ergodic}
    If $S \subset \mathbb{N}$ is syndetic and the word complexity of the
    sequence $1_S$ grows subexponentially (i.e. the orbit closure of $1_S$ under the left shift has zero entropy), then $S$ is not a weak interpolation set of order $2$ for uniquely ergodic systems.
\end{Proposition}

\medskip

\noindent \textbf{Outline of paper.}
\cref{sec_background} contains some background and general discussion. Sections \ref{sec_mixing}-\ref{sec:strictly_ergodic} concern, respectively, totally transitive/weak mixing/mixing systems, systems with bounded entropy, and strictly ergodic systems of zero entropy. In \cref{sec:distal} we prove some results about distal systems, including an answer to a question of Host, Kra, and Maass 
\cite{Host-Kra-Maass-2016} on the connection between sets of pointwise recurrence for distal systems and $IP$-sets.
Finally, \cref{sec_Qs} contains some open questions that naturally arise from our study.

\section{Background}
\label{sec_background}

\subsection{Some families of subsets of integers}
\begin{Definition}\label{synd}
Let $S$ be a nonempty subset of $\N$.
\begin{itemize}
    \item[$\bullet$] $S$ is \emph{syndetic} if the gaps between consecutive elements of $S$ are bounded. Equivalently, $S$ is syndetic if there exists $k \in \N$ such that $S \cup (S - 1) \cup \ldots \cup (S - k) \supset \N$.

    \item[$\bullet$] $S$ is \emph{thick} if it contains arbitrary long intervals of the form $\{m, m+1, \ldots, n\}$.

    \item[$\bullet$] $S$ is \emph{piecewise syndetic} if it is the intersection of a syndetic set and a thick set.

    \item[$\bullet$] $S$ is \emph{thickly syndetic} if for every $n\in\N$, there exists a syndetic set $E$ such that $\bigcup_{m \in E} \{m, m + 1, \ldots,  m + n\} \subset S$. 

    The classes of thick sets and syndetic sets are dual in the sense that a set $S$ is syndetic if and only $S$ intersects every thick set.
    Similarly, the classes of thickly syndetic and piecewise syndetic sets are dual in the same sense.
\end{itemize}
\end{Definition}

\begin{Definition}

The \emph{upper Banach density of $S \subset \N$} is
\[
    d^*(S) = \sup_{(F_N)_{N \in \N}} \limsup_{N \to \infty} \frac{|S \cap F_N|}{|F_N|}
\]
where the supremum is taken over all sequences of intervals $(F_N)_{N \in \N}$ in $\N$ satisfying $|F_N| \to \infty$ as $N \to \infty$. An equivalent definition of the upper Banach density of $S$ is
    \[
    d^*(S) = \lim_{n \to \infty} \sup_{m \in \N} \frac{|S \cap \{m, \ldots, m+n-1\}|}{n}.
    \]   
If $d^*(S) = 0$, then $\limsup_{N \to \infty} \frac{|S \cap F_N|}{|F_N|} = 0$ for any sequence $(F_N)_{N \in \N}$ of intervals in $\N$ whose lengths tend to infinity. In this case, we say $S$ has \emph{zero Banach density}.
\end{Definition}

\subsection{Some classes of dynamical systems}

\begin{Definition}
A \emph{topological dynamical system} (or a \emph{system} for short) is a pair $(X,T),$ where $X$ is a compact metric space 
and $T:X\to X$ is a continuous map.

A point $x\in X$ is \emph{transitive} if its forward orbit $o_T(x):=\{T^n x:\; n \in \N_0\}$ is dense in $X$.
\end{Definition}

\begin{Definition} \label{def:transitive-mixing-etc}

A system $(X,T)$ is
\begin{itemize}
    \item[$\bullet$] \emph{transitive} 
    if it contains a transitive 
    point, 

    \medskip

    \item[$\bullet$] \emph{totally transitive} if $(X,T^k)$ is transitive for every $k\in \N$,  

    \item[$\bullet$] (topological) \emph{weakly mixing} if for any open $U, V\subset X$, the set $\{n\in\N:\; T^{-n}U\cap V\neq \varnothing\}$ is thick, 

    \item[$\bullet$] (topological) \emph{mixing} if for any open $U, V\subset X$, the set $\{n\in\N:\; T^{-n}U\cap V\neq \varnothing\}$ is co-finite,

    \item[$\bullet$] \emph{uniquely ergodic} if there exists a unique Borel probability measure $\mu$ on $X$ satisfying $\mu(A) = \mu(T^{-1}A)$ for every Borel set $A \subset X$,

\item[$\bullet$] \emph{totally uniquely ergodic} if $(X, T^k)$ is uniquely ergodic for every $k\in\N$.
\end{itemize}
\end{Definition}

\begin{Remark}
    There are several other notions of transitivity in topological dynamics. Among them, the most commonly used is probably the following: A system $(X, T)$ is \emph{set transitive} if for any open $U, V$, there exists $n \in \N$ such that $U \cap T^{-n} V \neq \varnothing$; and $(X, T)$ is \emph{set totally transitive} if $(X, T^k)$ is set transitive for all $k \in \N$. 
    
    When $X$ is a compact metric space, set transitivity implies transitivity -- the notion we introduce in \cref{def:transitive-mixing-etc}. Likewise, set total transitivity implies total transitivity. 
    The converse implications of both statement also hold if $X$ has no isolated point. (See \cite{Akin-Carlson-topological-transitivity} for a discussion on various notions of transitivity and their relations.)
\end{Remark}






\begin{Definition}\label{def:entropy}
Let $(X,T)$ be a system and let $\rho$ be a metric on $X$. Let $\varepsilon>0$ and denote by $c(n,\varepsilon,T)$ the smallest cardinality of a cover of $X$ by sets with $\rho_n$-diameter less than $\varepsilon,$ where $\rho_n(x,y)=\max_{0\leq k\leq n}\rho(T^k x, T^k y)$. 
 The \emph{topological entropy} of the system is the (finite) number
\begin{equation*}
h(T)= \lim_{\varepsilon\to 0^+}\lim_{n\to\infty}\frac{\log (c(n,\varepsilon,T))}{n}.
\end{equation*}
\end{Definition}

\begin{Definition}
    Let $\mathcal{A}$ be a compact set and define $\Omega_{\mathcal{A}} = \mathcal{A}^{\N_0}$ endowed with the product topology. Every point $x \in \Omega_{\mathcal{A}}$ is a sequence $(x(n))_{n \in \N_0}$ with $x(n) \in \mathcal{A}$. Define the \emph{left shift} $\sigma: \Omega_{\mathcal{A}} \to \Omega_{\mathcal{A}}$ by $\sigma((x(n))_{n \in \N_0}) = (x(n+1))_{n \in \N_0}$. 
    
    The dynamical system $(\Omega_{\mathcal{A}}, \sigma)$ is mixing and is called \emph{the full shift on $\mathcal{A}$}. A \emph{subshift on $\mathcal{A}$} is a pair $(X, \sigma)$ where $X$ is a closed subset of $X$ satisfying $\sigma(X) \subset X$. 
\end{Definition}

\subsection{Equivalence of strong interpolation and weak interpolation for some classes of systems}
\label{sec:equivalence}

In this section, we show the equivalence of various notions of interpolation sets for certain classes of systems. 
In the special case where $\mathcal{C}$ is the class of compact abelian group rotations, this result was proved in \cite[Theorem 1]{Hartman_Ryll-Nardzewski_1964}. 
An analogous result for nilsystems was stated in \cite[Theorem 2.2]{Le_sublac} without proof. 
Our proof here is similar to the one in \cite[Theorem 1]{Hartman_Ryll-Nardzewski_1964}.

\begin{Proposition}
If $\mathcal{C}$ is a class of dynamical systems satisfying
\begin{enumerate}
    \item[(1)] If $(X, T) \in \mathcal{C}$ and $(Y, T)$ is a subsystem of $(X, T)$, then $(Y, T) \in \mathcal{C}$ and
    \item[(2)]\label{item_somelabelname} If $(X, T),(Y, R) \in \mathcal{C}$, then $(X \times Y, T \times R) \in \mathcal{C}$,
\end{enumerate}
then weak interpolation of order $2$ for $\mathcal{C}$ is equivalent to weak interpolation of all orders for $\mathcal{C}$. If in addition $\mathcal{C}$ satisfies
\begin{enumerate}
\item[(2')] If $(X_n, T_n) \in \mathcal{C}$ for $n \in \mathbb{N}$, then $(\prod_n X_n, \prod_n T_n) \in \mathcal{C}$,
\end{enumerate}
then weak interpolation of order $2$ for $\mathcal{C}$ is equivalent to strong interpolation for $\mathcal{C}$.
\end{Proposition}

\begin{Remark}
We note that (1) and (2) hold for the classes of expansive systems and systems of finite entropy, and that (1) and (2') hold for the classes of systems of zero entropy, compact abelian group rotations, inverse limits of nilsystems, and distal systems.  

On the other hand, (1) fails for several classes of systems, eg., the class of systems with positive entropy, the class of transitive systems and the class of (weak) mixing system. 
Condition (2) fails for the class of minimal systems and the class of uniquely ergodic systems.
\end{Remark}

\begin{proof}
Assume that $\mathcal{C}$ satisfies (1) and (2). Since weak interpolation of order $k$ implies weak interpolation of all smaller orders, it suffices to show that weak interpolation of order $2$ implies weak interpolation of order $2^m$ for all $m$. 
Suppose that $S\subset\N$ is a weak interpolation set of order $2$ for $\mathcal{C}$ and let $f: S \to \{0,1, \ldots, 2^m-1\}$ be arbitrary. 
We can decompose $f = \sum_{k=0}^{m-1} f_k$ with $f_k: S \to \{0, 2^k\}$ by binary expansion. Then, for each $k$, since $S$ is weak interpolation of order $2$, there exists $(X_k, T_k) \in \mathcal{C}$, a continuous function $F_k: X_k \to [0, 1]$ and a point $x_k \in X_k$ such that 
\[
    F_k(T_k^n x_k) = f_k(n) \text{ for all } n \in S.
\]
Define the orbit closure $Y$ of the point $x := (x_k)_{0 \leq k < m}$ in the product system 
$\prod_{k = 0}^{m-1} (X_k, T_k)$. 
Denote $T = \prod_{k=0}^{m-1} T_k$. 
By (1) and (2), since $(Y, T)$ is a subsystem of a finite product of $(X_k, T_k)$, $(Y, T) \in \mathcal{C}$. 
Finally, define the continuous function 
$F \in C\left(\prod_{k=1}^{\infty} X_k\right)$ as $F(y_0, \ldots, y_{m-1}) = \sum_{k=0}^{m-1} 2^k F_k(y_k)$ and note that
$F(T^n x) = f(n)$ for $n \in S$.
This shows that $S$ is weak interpolation of order $2^m$ for $\mathcal{C}$, and since $m$ was arbitrary, of all orders. 

The proof of the second part is similar. 
Again, we need only prove that if $\mathcal{C}$ satisfies (1) and (2') and $S\subset\N$ is weak interpolation of order $2$ for $\mathcal{C}$, then it is interpolation for $\mathcal{C}$. 
Consider any such $S$ and any bounded function $f: S \to \mathbb{C}$. 
By the previous part of the proof, we know that $S$ is weak interpolation of all orders.
Letting $C := \|f\|_{\infty}$, we can write $f = C \left(\sum_{k=1}^{\infty} f_k + ig_k\right)$ with $f_k, g_k: S \to \{0, \pm 2^{-k}\}$. 
Then we proceed exactly as above, defining $(X_k, T_k), (Y_k, S_k) \in \mathcal{C}$, $x_k \in X_k$, $y_k \in Y_k$, $F_k \in C(X_k)$, and $G_k \in C(Y_k)$ so that
\[
    F_k(T_k^n x_k) = f_k(n), G_k(S_k^n y_k) = g_k(n) \text{ for all } n \in S.
\]
Now we again take the orbit closure of the sequence $((x_k, y_k))_{k \in \mathbb{N}}$ in the infinite product system, which is in $\mathcal{C}$ by (1) and (2'). The function $F: \prod_{k=1}^{\infty} X_k \to \mathbb{C}$ defined by 
$F((x_k, y_k))_{k \in \N}) = C \sum_{k=1}^{\infty} (F_k(x_k) + iG_k(y_k))$ is a uniform limit of continuous functions, and therefore continuous. Just as before, $F(T^n x) = f(n)$ for all $n \in S$, verifying that $S$ is an interpolation set for $\mathcal{C}$.
\end{proof}

\section{Totally transitive, weak mixing, and strong mixing systems}\label{sec_mixing}

In this section we prove \cref{thm:total_transitive_weak_mixing_mixing}.
We start by establishing some properties of totally transitive points.

\begin{Definition}
Given a totally transitive system $(X,T)$, a point $x\in X$ is a \emph{totally transitive point} if $\{T^{kn} x: n \in \N_0\}$ is dense in $X$ for every $k \in \N$.
\end{Definition}

A transitive system always has a transitive point, but it is not immediately obvious that every totally transitive system has a totally transitive point. 
Nevertheless, the next lemma shows that this is indeed the case.

\begin{Lemma}\label{prop:dichotomy}
Let $(X,T)$ be totally transitive. Then one of the following two mutually exclusive statements is true.
\begin{enumerate}
\item \label{item:transitive-surjective} $T: X \to X$ is surjective and there is a residual set of totally transitive points.
\item \label{item:transitive-isolated-point}
There exists exactly one transitive point $x\in X$, which is an isolated point of $X$. 
In this case $x$ is a totally transitive point, $(X\setminus\{x\},T)$ is a totally transitive subsystem of $(X,T)$, and $T:X\setminus\{x\}\to X\setminus\{x\}$ is surjective. 
\end{enumerate}
\end{Lemma}

\begin{proof}


Suppose $T$ is surjective, i.e. $T X = X$. 
If $x$ is a transitive point of $(X,T)$, then 
\[
    \overline{\{T^{n} (Tx): n \in \N_0\}} = T \overline{\{T^{n} x: n \in \N_0\}} = TX = X,
\]
and so $Tx$ is also a transitive point.
In particular, whenever $(X,T)$ is transitive and $T$ is surjective, there exists a dense set of transitive points.
Surjectivity of $T$ implies surjectivity of $T^k$, so the system $(X,T^k)$ contains a dense set of transitive points.

If $x\in X$ is transitive for $T^k$, then for every non-empty open set $U\subset X$ there exists some $n\in\N_0$ with $x\in T^{-nk}U$, so the union $\bigcup_{n\in\N_0}T^{-nk}U$ contains every transitive point and hence is an open dense set.
Taking now a countable basis $(U_\ell)_{\ell\in\N}$ for $X$, it follows that the intersection ${\mathcal T}:=\bigcap_{\ell,k\in\N}\bigcup_{n\in\N_0} T^{-kn}U_\ell$ is a countable intersection of dense open sets and hence is residual.
To finish the proof, notice that any point in ${\mathcal T}$ is totally transitive, since for every $k\in\N$ and every non-empty open set $U\subset X$ there exists $\ell\in\N$ with $U_\ell\subset U$ and hence there exists some $n\in\N_0$ such that $T^{nk}x\in U_\ell\subset U$.

(\ref{item:transitive-isolated-point}) If $(X,T)$ is a transitive system that is not surjective, then we claim that $U:=X\setminus TX$ has exactly one element.
Indeed, $TX$ is compact, hence $U$ is open. 
By transitivity, there exists a point $x\in X$ with a dense orbit, so $T^nx\in U$ for some $n\geq0$; the definition of $U$ forces $n=0$ and hence $x\in U$.
If there were some $y\in U$ with $y\neq x$ we could take $V\subset U$ to be a neighborhood of $y$ not containing $x$. 
The orbit of $x$ never enters $V$ which is a contradiction, and this proves the claim that $U=\{x\}$.


Let $X_0 = X \setminus \{x\} = T X$. Let $k \in \N$ and $y = T^k x$. 
Then $y \in X_0$. 
Since $x$ is a totally transitive point and $x$ is isolated, 
\[
    \overline{\{T^{kn} y: n \geq 0\}} = \overline{\{T^{kn} x: n \geq 1\}} = X_0,
\]
showing that $y$ is a transitive point for $(X_0,T^k)$, and in particular that $(X_0,T^k)$ is a transitive system. 
Since $k$ is arbitrary, $(X_0, T)$ is totally transitive. 

Since $X_0 = TX$, for $k \geq 2$, we have $\{T^{kn} x: n \geq 1\} \subset T X_0$. Because the set in the left hand side is dense in $X_0$, $T X_0$ is dense in $X_0$. As both $T X_0$ and $X_0$ are closed, we conclude that $T X_0 = X_0$, i.e. $T$ is surjective on $X_0$.
\end{proof}



\begin{lemma}
\label{lemma_totaltransitivity_new}
If $(X,T)$ is totally transitive then every transitive point of $(X, T)$ is a totally transitive point.
\end{lemma}

\begin{proof}
If $T$ is not surjective, then by \cref{prop:dichotomy}, $(X, T)$ has exactly one transitive point. Since $(X, T)$ is totally transitive, this unique transitive point is also a totally transitive point. 

Next, assume that $T$ is surjective and let $x$ be a transitive point of $(X,T)$.
Let $k\geq 2$ and take $Y=\overline{\{T^{kn}x: n\in\N_0\}}$.
Since $(X,T^k)$ is transitive, there exists a point $y\in X$ that is $T^k$-transitive. 
Since $x$ is transitive, it follows that $Y\cup T^{-1}Y\cup\ldots \cup T^{-(k-1)} Y =X$ and hence $T^j y\in Y$ for some $j\in\{0,1,\ldots,k-1\}$. 
Since $y$ is $T^k$-transitive and $T$ is surjective, $T^j y$ is also $T^k$-transitive. 
It follows that $Y$ contains a $T^k$-transitive point of $X$, and is invariant under $T^k$, so $Y=X$, proving that $x$ is $T^k$-transitive.
\end{proof}

The next lemma is the main technical step towards proving \cref{thm:total_transitive_weak_mixing_mixing}.

\begin{lemma}
\label{lem:syndetic_not_int2}
Let $S \subset \N$ be a syndetic set and let $k \in \N$ be such that $S \cup (S - 1) \cup \ldots \cup (S - (k-1)) \supset \N$. 
Let $h > k$ and $S_i := S \cap (h^2\N + [ih, (i+1)h))$ for $0 \leq i \leq h - 1$. For any totally transitive system $(X, T)$ and transitive point $x \in X$, there exists $0 \leq i < j \leq h - 1$ such that 
\[
    \overline{\{T^n x: n \in S_i\}} \cap \overline{\{T^n x: n \in S_j\}} \neq \varnothing.
\]
\end{lemma}

\begin{proof}
Since $S_i$ is nonempty for all $i$, the conclusion is obviously true if $(X, T)$ is the trivial one-point system. Now suppose $(X, T)$ is a nontrivial, totally transitive system and let $x$ be a transitive point of $(X, T)$. 
By \cref{lemma_totaltransitivity_new}, the point $x$ is a totally transitive point. 
We then have that $T^{a \N + b}x$ is dense in $X \setminus \{x\}$ for any $a, b \in \N$ (where, for any set $S \subset \N_0$, $T^S x := \{T^n x: n \in S\}$); this can be seen by analysing the two cases arising in \cref{prop:dichotomy} individually.   



For $0 \leq i \leq h - 1$, let $X_i := \overline{T^{S_i} x}$. Note that for all $0 \leq i \leq h -1$ and $0 \leq j \leq k - 1$,
\[
    h^2 \N + [ih, (i+1)h) - j \supset h^2 \N + ih.
\]
Therefore, for $0 \leq i \leq h -1$ and $0 \leq j \leq k - 1$,
\[
    S_i - j = (S \cap (h^2 \N + [ih, (i+1)h))) - j \supset (S-j) \cap (h^2 \N + ih).
\]
And so,
\begin{align*}
    S_i \cup (S_i - 1) \cup \ldots \cup (S_i - (k-1)) &\supset (S \cup (S - 1) \cup \ldots \cup (S - (k-1))) \cap (h^2 \N + ih)\\ &= h^2 \N + ih.
\end{align*}
It follows that
\begin{equation}
\label{eq:X_i_align}
\begin{aligned}
    X_i \cup T^{-1} X_i \cup \ldots \cup T^{-(k-1)} X_i & \supset \overline{T^{S_i} x} \cup \overline{T^{S_i - 1} x} \cup \ldots \cup\overline{T^{S_i - (k-1)} x}\\
    &\supset \overline{T^{h^2 \N + ih} x}\\
    &\supset X \setminus \{x\}.
\end{aligned}
\end{equation}

Let $\mu$ be a $T$-invariant measure on $X$. 
Since $(X, T)$ is totally transitive and nontrivial, it is not a finite system, and since $x$ is a transitive point, it is not a periodic point. Therefore, $\mu(\{x\}) = 0$ and $\mu(X \setminus \{x\}) = 1$. Then \eqref{eq:X_i_align} implies $\mu(X_i) \geq 1/k$ for all $0 \leq i \leq h - 1$. Since $h > k$, there must be $0 \leq i < j \leq h - 1$ such that $X_i \cap X_j \neq \varnothing$ which is the desired conclusion.
\end{proof}

We are ready to prove \cref{thm:total_transitive_weak_mixing_mixing} whose statement is repeated here for convenience.

    
    
    


\begin{named}{\cref{thm:total_transitive_weak_mixing_mixing}}{}
Let $S \subset \N$. The following are equivalent:
\begin{enumerate}
    \item \label{item:equivalent_weak_totaltransitive_of_all_order} $S$ is a weak interpolation set of all orders for totally transitive systems.
    
    \item \label{item:equivalent_totaltransitive} $S$ is an interpolation set for totally transitive systems.
    
    \item \label{item:equivalent_weakmixing} $S$ is an interpolation set for weak mixing systems.
    
    \item \label{item:equivalent_mixing} $S$ is an interpolation set for strong mixing systems.

    \item \label{item:equivalent_not_synd} $S$ is not syndetic.

\end{enumerate}
\end{named}

\begin{proof}

We have the chain of implications among various classes of topological dynamical systems: Strong mixing $\Rightarrow$ Weak mixing $\Rightarrow$ Totally transitive. Therefore, the following directions are obvious: (\ref{item:equivalent_mixing}) $\Rightarrow$ (\ref{item:equivalent_weakmixing}) $\Rightarrow$ (\ref{item:equivalent_totaltransitive}) $\Rightarrow$ (\ref{item:equivalent_weak_totaltransitive_of_all_order}) and so it remains to prove (\ref{item:equivalent_weak_totaltransitive_of_all_order}) $\Rightarrow$ (\ref{item:equivalent_not_synd}) and (\ref{item:equivalent_not_synd}) $\Rightarrow$ (\ref{item:equivalent_mixing}).

(\ref{item:equivalent_weak_totaltransitive_of_all_order}) $\Rightarrow$ (\ref{item:equivalent_not_synd}): For the sake of contradiction, suppose $S$ is syndetic. Let $k \in \N$ be such that $\bigcup_{i=0}^{k-1} (S - i) = \N$. Fix an arbitrary integer $h > k$ and define $S_i := S \cap (h^2\N + [ih, (i+1)h))$ for $0 \leq i \leq h - 1$.
Let $f: S \to \R$ be a bounded function that satisfies $f(n) = i$ for $n \in S_i$.

By \cref{lem:syndetic_not_int2}, for any totally transitive system $(X, T)$ and transitive point $x \in X$, there exists $0 \leq i < j \leq h - 1$ such that $\overline{T^{S_i} x} \cap \overline{T^{S_j} x} \neq \varnothing$. Hence, there does not exist $F \in C(X)$ such that $F(T^{S_i} x) = \{i\}$ and $F(T^{S_j} x) = \{j\}$. As a result, there is no $F \in C(X)$ such that the function $n \mapsto F(T^nx)$ extends $f$. In particular, $S$ is not weak interpolation of order $h$, and so also not weak interpolation of all orders, for totally transitive systems.

(\ref{item:equivalent_not_synd}) $\Rightarrow$ (\ref{item:equivalent_mixing}): 
Let $S\subset\N$ be a set which is not syndetic, take an arbitrary bounded function $f:S\to\C$ and let $K\subset\C$ be a compact set containing $f(S)$.
We consider the full shift with (potentially infinite) alphabet $K$, defined as the topological dynamical system $(K^{\N_0},\sigma)$ where $\sigma:(z_n)_{n\in\N_0}\mapsto(z_{n+1})_{n\in\N_0}$ is the left shift.
Note that $K^{\N_0}$, equipped with the product topology is a compact metric space and that $\sigma$ is continuous.
The system $(K^{\N_0},\sigma)$ is strong mixing and, in particular, transitive.
Let $y\in K^{\N_0}$ be a transitive point.
We think of points in $K^{\N_0}$ as functions $\N_0\to K$.

Since $S$ is not syndetic, then for every $n\in\N$ there is an interval $I_n:=\{m_n+1,m_n+2,\dots,m_n+n\}$ of length $n$ which is disjoint from $S$.
Extend $f:S\to K$ to a function $g:\N_0\to K$ by letting $g|_{I_n}=y|_{\{1,\dots,n\}}$.
Then the orbit closure of $g$ as an element of $K^{\N_0}$ contains $y$, and since $y$ is a transitive point, so is $g$.
Finally, let $F:K^{\N_0}\to\C$ be the projection onto the $0$-th coordinate.
If follows that $F(\sigma^ng)=g(n)$ for all $n\in\N_0$.
In particular, $F(\sigma^ng)=f(n)$ for all $n\in S$, showing that $S$ is an interpolation set for strongly mixing systems.
\end{proof}

\section{Minimal and totally minimal systems}

In this section we prove \cref{thm:interpolation_minimal_systems_new} by combining the next two lemmas. The first lemma comes from \cite{Glasner-Weiss-Interpolation}.

\begin{Lemma}[Glasner-Weiss {\cite[Theorem 1 (Part 1)]{Glasner-Weiss-Interpolation}}]
\label{lem:lem:glasner-weiss-order-2}
Piecewise syndetic sets are not weak interpolation of order $2$ for minimal systems.
\end{Lemma}
\begin{Remark} 
The result of Glasner and Weiss given in \cite{Glasner-Weiss-Interpolation} asserts that any piecewise syndetic set is not an interpolation set for minimal systems. 
However, their proof yields the stronger conclusion after careful examination. 
Indeed, their proof involves constructing, for any piecewise syndetic set $S$, a function $\eta$ with domain $S$ for which $\eta(s)$ cannot be represented as $F(T^s x)$ for $s \in S$ and a minimal $(X, T)$. This function $\eta$ always has codomain $\{0,1\}$, and so this immediately also demonstrates that $S$ is not weak interpolation of order $2$ for minimal systems.    
\end{Remark}

\begin{Lemma}\label{lem:non-PW-syndetic-totally-minimal}
Every non-piecewise syndetic set is an interpolation set for totally minimal systems.
\end{Lemma}

\begin{proof}
Let $S\subset\N$ be a non-piecewise syndetic set.
Let $u: S\to\C$ be an arbitrary bounded function and let $K\subset\C$ be a compact set containing $u(S)$ and $0$.
We construct a compact set $X\subset K^{\N_0}$ iteratively via auxiliary sequences $m_k \in \mathbb{N}$, closed (thereby compact) sets 
$X_k \subset K^{m_k}$ and $X'_k \subset K^{m_k + 1}$, and
$w_k \in X_k$. 
Define $m_0 = 1$, $X_0 = K$, $X'_0 = K^2$, and $w_0 = 0$. 
Suppose now that $k\geq0$ and that $m_k$, $X_k$, and $w_k$ are defined. 
Since $X_k, X'_k$ are compact,
there exist $2^{-k}$-dense 
finite subsets $T_k, T'_k$ of $X_k$ and $X'_k$ respectively (in the $\ell^\infty$ metric). 
Without loss of generality, we assume that $k!$ divides both $|T_k|$ and $|T'_k|$.
Define
$m_{k+1}$ to be a multiple of $m_k (k+1)!$ 
with the property that every interval of length 
$m_{k+1}$ contains a gap in $S$ of length $4m_k^2(|T_k| + |T'_k|)$,
i.e. for every interval $I\subset\N$ of length $m_{k+1}$, there exists a subinterval 
$J\subset I$ of length
$4m_k^2(|T_k| + |T'_k|)$ so that $J \cap S = \varnothing$. 
(We use here the fact that $S$ is not piecewise syndetic, which is equivalent to the fact that for all $n$, gaps of length $n$ occur syndetically in $S$.) Note that by construction, $k!$ divides $m_k$ for all $k$.

Define $X_{k+1}$ to be the set of all strings $w$ with length $m_{k+1}$ which are concatenations of strings in $X_k$ and $X'_k$ and which have the following property: for each $0 \leq i < k!$, the set of concatenated elements of $X_k$ whose first letter is at a location in $w$
equal to $i \pmod{k!}$ is $2^{-k}$-dense in $X_k$, and the set of concatenated elements of $X'_k$ whose first letter is at a location in $w$
equal to $i \pmod{k!}$ is $2^{-k}$-dense in $X'_k$. 
Similarly define $X'_{k+1}$, with the only change being that elements of $X'_{k+1}$ must have length $m_{k+1} + 1$.

We note that both $X_{k+1}$ and $X_{k+1}'$ are nonempty. 
Indeed, if $T_k = \{w^{(k)}_1, \ldots, w^{(k)}_{|T_k|}\}$, $T'_k = \{w'^{(k)}_1, \ldots, w'^{(k)}_{|T'_k|}\}$, and $v$ is an arbitrary element of $X'_k$, then the reader may check that the string
\[
w_{k+1} := w_k w_k \ldots w_k (w^{(k)}_1 w^{(k)}_2 \ldots w^{(k)}_{|T_k|} w'^{(k)}_1 w'^{(k)}_2 \ldots w'^{(k)}_{|T'_k|} v)
\ldots
(w^{(k)}_1 w^{(k)}_2 \ldots w^{(k)}_{|T_k|} w'^{(k)}_1 w'^{(k)}_2 \ldots w'^{(k)}_{|T'_k|} v)
\]
is in $X_{k+1}$, where $w^{(k)}_1 w^{(k)}_2 \ldots w^{(k)}_{|T_k|} w'^{(k)}_1 w'^{(k)}_2 \ldots w'^{(k)}_{|T'_k|} v$ (which has length equal to $1 \pmod{k!}$) is repeated $m_k$ times and the number of $w_k$ at the beginning is chosen to make the entire word have length $m_{k+1}$. (Note that the terminal portion has length $m_k\left(m_k|T_k| + (m_k+1)(|T'_k|+1)\right) < 4m_k^2 (|T_k| + |T'_k|)$, which is less than $m_{k+1}$.) Similarly,
\[
    v w_k w_k \ldots w_k (w^{(k)}_1 w^{(k)}_2 \ldots w^{(k)}_{|T_k|} w'^{(k)}_1 w'^{(k)}_2 \ldots w'^{(k)}_{|T'_k|} v)
    \ldots
    (w^{(k)}_1 w^{(k)}_2 \ldots w^{(k)}_{|T_k|} w'^{(k)}_1 w'^{(k)}_2 \ldots w'^{(k)}_{|T'_k|} v)
\]
is in $X'_{k+1}$, where the first concatenated $w_k\in X_k$ above is replaced by $v\in X_{k}'$.

Since $w_k$ is a prefix of $w_{k+1}$ for all $k$, we can define $y$ to be the limit of the $w_k$, and define $X$ to be the closure of the orbit of $y$. 

We claim that for any $j > 0$, $y$ is uniformly recurrent under the shift action by $\sigma^j$. 
Indeed, for any neighborhood $U$ of $y$, there exists $k > j$ so that
$V = \{z \ : \ \rho(z_1 \ldots z_{m_k}, y_1 \ldots y_{m_k}) < 2^{-k}\} \subset U$. (Here $\rho$ denotes the metric in $X_k$.)  
By definition, $w_k = y_1 \ldots y_{m_k} \in X_k$. 
Also by definition, $y$ is a concatenation of elements of $X_{k+1}$ and $X'_{k+1}$, meaning that the set of 
locations $t$ for which either $y(t) \ldots y(t + m_{k+1} - 1) \in X_k$ or $y(t) \ldots y(t + m_{k+1}) \in X'_k$ is syndetic (with gaps bounded from above by $m_{k+1} + 1$). 
For each such $t$, there exists a location $s \equiv -t \pmod{k!}$ so that 
$\rho(y(t+s) y(t+s+1) \ldots y(t+s+m_k -1), w_k) < 2^{-k}$, i.e.
$\sigma^{t+s} y \in V \subset U$. 
In addition, $t+s$ is a multiple of $k!$, and so a multiple of $j$. This implies that the set $\{n \in \N: \sigma^{jn} y \in U\}$ is syndetic.
Since $U$ was arbitrary, $y$ is uniformly recurrent for $\sigma^j$. Since $j$ was arbitrary, $X$ is totally minimal. 

It remains to construct $x_u \in X$ for which $x_u(s) = u(s)$ for all $s \in S$.
Let $S_0=S$ and, for each $k\in\N$ let $S_k$ be the union of all intervals of the form $(im_k,(i+1)m_k]$ which are not disjoint from $S$.
Note that $S_0\subset S_1\subset\cdots$ and that $\bigcup S_k=\N$.
For $k=0$ we take $x^{(0)}(s)=u(s)$ for all $s\in S=S_0$.
Then, for each $k\in\N_0$, we will inductively construct a function $x^{(k)}:S_k\to K$ so that the restriction of $x^{(k)}$ to $S_{k-1}$ coincides with $x^{(k-1)}$, and for each interval $(im_k,(i+1)m_k]$ the restriction $x^{(k)}\big((im_k,(i+1)m_k]\big)$ is a word from $X_k$.

Suppose now that $x^{(k)}:S_k\to K$ has been constructed.
Fix an interval $(im_{k+1},(i+1)m_{k+1}]$ contained in $S_{k+1}$.
From the definition of $m_{k+1}$, this interval $(im_{k+1},(i+1)m_{k+1}]$ contains a subinterval $J_0$ of length $4m_k^2\big(|T_k|+|T_k'|\big)$ which is disjoint from $S$.
In particular $J_0$ contains the interval $J=(am_k,bm_k]$ for some $a,b\in\N$ satisfying $b-a=m_k|T_k|+(m_k+1)(|T_k'|+1)$, and hence $J$ is disjoint from $S_k$.
We can then construct $x^{(k+1)}:(im_{k+1},(i+1)m_{k+1}]\to K$ so that 
\begin{itemize}
    \item $x^{(k+1)}\big((jm_k,(j+1)m_k]\big)=x^{(k)}\big((jm_k,(j+1)m_k]\big)\in X_k$ whenever $(jm_k,(j+1)m_k]\subset S_k$;
    \item$\displaystyle
x^{(k+1)}(J)=(w^{(k)}_1 w^{(k)}_2 \ldots w^{(k)}_{|T_k|} w'^{(k)}_1 w'^{(k)}_2 \ldots w'^{(k)}_{|T'_k|} v)
\ldots
(w^{(k)}_1 w^{(k)}_2 \ldots w^{(k)}_{|T_k|} w'^{(k)}_1 w'^{(k)}_2 \ldots w'^{(k)}_{|T'_k|} v),
$
where the string $w^{(k)}_1 w^{(k)}_2 \ldots w^{(k)}_{|T_k|} w'^{(k)}_1 w'^{(k)}_2 \ldots w'^{(k)}_{|T'_k|} v$ is repeated $m_k$ times.
\item $x^{(k+1)}\big((jm_k,(j+1)m_k]\big)\in X_k$ is arbitrary whenever $(jm_k,(j+1)m_k]\not\subset J$.
\end{itemize}
The conditions above imply that $x^{(k+1)}\big((im_{k+1},(i+1)m_{k+1}]\big)\in X_{k+1}$, so using this construction in each interval $(im_{k+1},(i+1)m_{k+1}]$ in $S_{k+1}$ we have constructed $x^{(k+1)}:S_{k+1}\to K$ as desired.

We can now take the limit as $k\to\infty$ to obtain $x_u:\N\to K$, formally defined as $x_u(s)=x^{(k)}(s)$ whenever $s\in S_k$; the construction above insures that $x_u$ is well defined.
Moreover, $x_u\big((0,m_k]\big)$ is in $X_k$ for every sufficiently large $k$, and so some shift of $y$ visits arbitrarily small neighborhoods of $x_u$, showing that $x_u\in X$ as desired.

Now, simply define $F \in C(X)$ by $F(x) = x(0)$. Then for all $s \in S$, $F(\sigma^s x_u) = x_u(s) = u(s)$. Since $u$ was arbitrary and $(X, \sigma)$ is totally minimal, we have shown that $S$ is an interpolation set for totally minimal systems. 
\end{proof}

\begin{named}{\cref{thm:interpolation_minimal_systems_new}}{}
Let $S \subset \N$. The following are equivalent:
\begin{enumerate}
    \item \label{item:weak-order-2} $S$ is a weak interpolation set of order $2$ for minimal systems. 
    
    \item \label{item:weak-all-order} $S$ is a weak interpolation set of all orders for minimal systems.
    
    \item \label{item:minimal} $S$ is an interpolation set for minimal systems.
    
    \item \label{item:totally-minimal} $S$ is an interpolation set for totally minimal systems.

    \item \label{item:not-PW} $S$ is not piecewise syndetic.
\end{enumerate}
\end{named}

\begin{proof}
The implications $(\ref{item:totally-minimal}) \Rightarrow (\ref{item:minimal}) \Rightarrow (\ref{item:weak-all-order}) \Rightarrow (\ref{item:weak-order-2})$ are obvious. The implication $(\ref{item:not-PW}) \Rightarrow (\ref{item:totally-minimal})$ is proved in \cref{lem:non-PW-syndetic-totally-minimal} and the implication $(\ref{item:weak-order-2}) \Rightarrow (\ref{item:not-PW})$ follows from \cref{lem:lem:glasner-weiss-order-2}.
\end{proof}

\section{Systems of bounded entropy}
\label{sec:entropy}

\begin{named}{\cref{prop:interpolation_finite_entropy}}{}
Let $S \subset \N$. The following are equivalent:
\begin{enumerate}
    \item \label{item:finite_entropy} $S$ is an interpolation set for systems of finite entropy.

    \item \label{item:entropy<=M} $S$ is an interpolation set for systems of entropy $\leq M$ for some fixed $M \geq 0$.

    \item \label{item:zero_entropy} $S$ is an interpolation set for systems of zero entropy.
    
    \item \label{item:entropy_and_zero_Banach_density} $d^*(S) = 0$.
\end{enumerate}
\end{named}
\begin{proof}
It is obvious that (\ref{item:zero_entropy}) $\Rightarrow$ (\ref{item:entropy<=M}) $\Rightarrow$ (\ref{item:finite_entropy}) and so it remains to prove (\ref{item:entropy_and_zero_Banach_density}) $\Rightarrow$ (\ref{item:zero_entropy}) and (\ref{item:finite_entropy}) $\Rightarrow$ (\ref{item:entropy_and_zero_Banach_density}).

(\ref{item:entropy_and_zero_Banach_density}) $\Rightarrow$ (\ref{item:zero_entropy}): 
This implication follows from Theorem 1.6, but we provide here a direct short proof.
Consider any set $S$ with $d^*(S) = 0$ and any bounded $f: S \rightarrow \mathbb{C}$. 
Extend $f$ to $\tilde{f}: \mathbb{N}_0 \rightarrow \mathbb{C}$ by defining $\tilde{f} = 0$ on $\N_0\setminus S$.
Now, the orbit closure of $\tilde{f}$ under the left shift $\sigma$ is
a shift system $(X, \sigma)$ with $X \subset \mathbb{C}^{\mathbb{N}_0}$. 
We claim that $(X,\sigma)$ has zero entropy.
Defining $F(x) = x(0)$ for all $x$ we have that for $n \in S$, $F(\sigma^n \tilde{f}) = \tilde{f}(n) = f(n)$, and $\tilde{f}$ is a transitive point in $X$ by definition, so assuming the claim we have that $S$ is an interpolation set for systems of zero entropy.

We claim that every $x \in X$ satisfies $d^*(x^{-1}(\C\setminus\{0\})) = 0$. Indeed, for any $\epsilon > 0$, by definition of upper Banach density, there exists $N$ so that for every $i \in \N$, 
\[
    \frac{|\{0 \leq j < N \ : \ \tilde{f}(i+j) \neq 0\}|}{N} < \epsilon.
\]
This property is preserved under shifts and limits, and so holds for any $x \in X$ as well, implying that $d^*(x^{-1}(\C\setminus\{0\})) = 0$.
Therefore, by the ergodic theorem, the only invariant measure on $(X, \sigma)$ is the delta-measure for the constant sequence $0$. By the variational principle, $(X, \sigma)$ has topological entropy $0$. 

(\ref{item:finite_entropy}) $\Rightarrow$ (\ref{item:entropy_and_zero_Banach_density}): Consider any set $S$ with $d^*(S) > 0$ and assume for a contradiction that $S$ is an interpolation set for the class of finite entropy systems. Choose $\delta < d^*(S)$. 


There exist arbitrarily large $m \in \N$ and sets $S_m \subset[1,m]$ with $|S_m| \geq \delta m$ for which infinitely many shifts of $S_m$ are contained in $S$. Write $S_m = \{s_{m,1}, \ldots, s_{m,|S_m|}\}$ in increasing order for all $m$.
Let $\mathcal{S}^1$ 
denote the unit circle $\{z \in \C: |z| = 1\}$. Construct a sequence $f:S\to \mathcal{S}^1$ which, for all $k$ and $m$, contains every possible sequence of $k$th roots of unity of length $|S_m|$ along some shift of $S_m$, i.e. for any such string $c = (c_1, \ldots, c_{|S_m|})$, there exists 
$j_c$ so that $f(j_c + s_i) = c_i$ for $1 \leq i \leq |S_m|$. 
By assumption, there exists a finite entropy system $(X, T)$, $x \in X$, and 
$F \in C(X)$ where $F(T^n x) = f(n)$ for all $n$. Say that $h(X, T) < \delta \log k$. Now, define $\epsilon$ so that points of $X$ within distance $\epsilon$ cannot have $F$-values which are distinct $k$th roots of unity. 
For every $m$ and every $|S_m|$-tuple $c = (c_1, \ldots, c_m)$ of $k$th roots of unity, $F(\sigma^{j_c+s_i} x) = f(j_c+s_i) = c_i$ for $1 \leq i \leq |S_m|$. This implies that the collection $\{\sigma^{j_c} x\}$ of cardinality $k^{|S_m|} \geq k^{\delta m}$ is 
$(m, \epsilon)$-separated, and since $m$ can be arbitrarily large, $h(X, T) \geq \delta \log k$, a contradiction.
\end{proof}

\begin{named}{\cref{prop:weak-finite-entropy}}{}
    Every subset $S \subset \N$ is a weak interpolation set of all orders for systems of finite entropy.
\end{named}
\begin{proof}
    We will prove that $\mathbb{N}$ itself is a weak interpolation set of all orders for the class of systems of finite topological entropy. Then, clearly, the same will be true for all subsets of $\N.$

    
    Consider any $k$ and any function $f: \mathbb{N}_0 \rightarrow \{0, \ldots, k-1\}$. The orbit closure of $f$ under the left shift is a subshift $(X, \sigma)$ that has entropy bounded by $\log k$. Define $F : X \to \{0, \ldots, k - 1\}$ by $F(x) = x(0)$ for all $x \in X$. It follows that $f$ is a transitive point of $(X, \sigma)$ and $F(\sigma^n f) = f(n)$ for all $n \in \N$.
\end{proof}

Recall the entropy function $H:[0,1]\to\R$ with $H(\delta)=-\big(\delta\log\delta+(1-\delta)\log(1-\delta)\big)$ for $\delta\in(0,1)$ and $H(0) = H(1) = 0$. It's well known that $H(\delta)$ represents the exponential growth rate of the number of words in $\{0,1\}^m$ with fewer than $\delta m$ appearances of the symbol 1.
For completeness we include the proof of the following corresponding estimate on larger alphabets.

\begin{Lemma}
    \label{conj_entropy41}
Fix $k \geq 2$. For each $m \in \N$ and $\delta \in[0,1/2]$, let 
$$S(m,\delta,k):=\Big\{f:\{0,\ldots, m-1\}\to\{0,\dots,k-1\}:\big|f^{-1}(0)\big| \geq(1-\delta)m\Big\}.$$
Then 
\[
    \displaystyle\lim_{m\to\infty}\frac{\log \big|S(m,\delta,k)\big|}{m} = \inf_{m}\frac{\log \big|S(m,\delta,k)\big|}{m} = H(\delta) + \delta \log(k-1).
\]
\end{Lemma}

\begin{proof}
Every element of $S(m, \delta, k)$ can be described by first describing the locations of the non-$0$ letters, and then assigning any non-$0$ letter at each of those locations. 
Therefore, 
\[
    |S(m, \delta, 2)| = \sum_{i = 0}^{\lfloor \delta m \rfloor} (k-1)^i \binom{m}{i}.
\]
Since $\delta \leq 1/2$, $\binom{m}{i}$ is increasing from $i = 0$ to $\lfloor \delta m \rfloor$ (and $(k-1)^i$ is increasing for all $i$), and so
\begin{equation}\label{deltabd}
(k-1)^{\lfloor \delta m \rfloor} {\binom{m}{ \lfloor \delta m \rfloor}} \leq |S(m, \delta, 2)| \leq (\lfloor \delta m \rfloor + 1) 
(k-1)^{\lfloor \delta m \rfloor} {\binom{m}{\lfloor \delta m \rfloor}}.
\end{equation}
Define $\delta_m = \frac{\lfloor \delta m \rfloor}{m}$. Clearly $\delta_m \leq 1/2$ and $\delta_m \rightarrow \delta$. By Stirling's approximation, 
\begin{align*}
\log(k-1)^{\lfloor \delta m \rfloor}  {\binom{m}{\lfloor \delta m \rfloor}}
&=  \log(k-1)^{\delta_m m} {\binom{m}{\delta_m m}}\\ 
&= 
\delta_m m \log(k-1) + \log m! - \log (\delta_m m)! - \log ((1 - \delta_m)m)!\\
&=\delta_m m \log(k-1) + m \log m - m - (\delta_m m) \log(\delta_m m) \\
&+ \delta_m m - ((1-\delta_m)m) \log((1 - \delta_m)m) + (1 - \delta_m) m + o(m)\\
&= \delta_m m \log(k-1) + m H(\delta_m) + o(1).
\end{align*}
Since $H$ is a continuous function and $\delta_m \rightarrow \delta$, by (\ref{deltabd}), 
$\frac{\log |S(m, \delta, k)|}{m} \rightarrow H(\delta) + \delta \log(k-1)$.

Finally, it's easily checked that the sequence $\log |S(m, \delta, k)|$ is subadditive in $k$, and so by Fekete's subadditivity lemma (see \cite[p.~233]{Fekete} or \cite[Theorem 4.9]{Walters82}), 
the limit of $\frac{\log |S(m, \delta, k)|}{m}$ is also the infimum.
\end{proof}

We recall~\cref{thm_entropydensityclassification_intro_new} from the introduction.
\begin{named}{\cref{thm_entropydensityclassification_intro_new}}{}
Let $\delta\in[0,1/2]$. For any $k \in \N$ and a set $S\subset\N,$ we have the following: 
\begin{enumerate}
\item \label{item:entropy-density-greater} If $d^*(S) > \delta$, then $S$ is not a weak interpolation set of order $k$ for systems of entropy $\leq \delta \log k$. 

\item \label{item:entropy-density-smaller}
If $d^*(S)\leq \delta$, then $S$ is a weak interpolation set of order $k$ for systems of entropy $\leq H(\delta) + \delta \log(k-1)$.

\end{enumerate}
Moreover, (1) is sharp in the following sense:

\begin{enumerate}
\item [3.] \label{item:density-sharp} There exists $S \subset \N$ with $d^*(S) = \delta$ such that $S$ is a weak interpolation set of order $k$ for the class of systems of entropy $\leq \delta \log k$.
\end{enumerate}
\end{named}


\begin{proof} \
(\ref{item:entropy-density-greater}) Assume $d^*(S) > \delta$ and take an intermediate value $\delta<\delta'<d^*(S)$. 
Then there exist arbitrarily large $m \in \N$ and infinitely many intervals $[n+1, n+m]$ such that
\[
    |S \cap [n + 1, n + m]| \geq \delta' m. 
\]
Since there are only finitely many subsets of $[1,m]$, there exists some set $S_m\subset[1,m]$ for which infinitely many shifts are contained in $S$.
Construct a sequence $f:S\to\{0,1, \ldots, k-1\}$ which, for all $m$, contains all possible functions on $S_m\to\{0,1,\ldots,k-1\}$ up to shift; note that there are at least $k^{\delta' m}$ many such words.

Suppose that $(X,T)$ is a topological dynamical system and that there exists a transitive point $x \in X$ and $\phi \in C(X)$ such that $\phi(T^n x) = f(n)$ for all $n \in S$. 
Note that the sets $S_i = \phi^{-1}(\{i\})$ are disjoint compact subsets of $X$ for $0 \leq i < k$, and so they are pairwise separated by some distance $\epsilon$. Then, for all $m$, the set $\{T^nx:n\in S\}$ contains an $(m, \epsilon)$-separated set of cardinality $k^{\delta' m}$. 
This implies that $h(X,T) \geq \delta'\log k>\delta\log k$.

In particular, $S$ is not a weak interpolation set of order $k$ for the class of systems of entropy $\leq \delta\log k$.

\medskip
 
(\ref{item:entropy-density-smaller}) Suppose $S \subset \N$ has $d^*(S) \leq \delta \leq 1/2$ and let $f: S \rightarrow \{0, \ldots, k-1\}$ be arbitrary. 
Extend $f$ to all of $\N_0$ by letting $f(n)=0$ whenever $n\notin S$. 
Identify $f$ with a point $x_f\in\{0,\dots,k-1\}^{\N_0}$ and let $X$ be its orbit closure under the left shift map.

  For any $\delta' > \delta$, by definition of $d^*$ there exists $M$ so that $|S \cap I| < |I|\delta'$ for all intervals of length at least $M$. 
  Therefore, for all $m \geq M$, the language set of $X$ satisfies $L_m(X) \subset S(m, \delta', k)$.
  Invoking \cref{conj_entropy41} we deduce that $h(X) \leq H(\delta') + \delta' \log(k-1)$. 
  Since $\delta' > \delta$ was arbitrary and $H$ is continuous, 
  $h(X) \leq H(\delta) + \delta \log(k-1)$, completing the proof.

\medskip
 
(3) The idea is to choose $S$ very regular so that even though $S$ has large density, we can always extend a function on $S$ to a function on $\N$ that has small entropy.

Let $S = \{\lfloor n/\delta \rfloor: n \in \N\}$ so that $d^*(S) = \delta$ and let $x\in\{0,1\}^{\N_0}$ be its indicator function.
The orbit closure $X$ of $x$ under the left shift map is a so-called Sturmian subshift with rotation number $\delta$. 
In particular, the number of $m$-letter words in the language of $X$ is $m+1$ for all $m$, and it's clear by definition that the number of $1$s in any word of length $m$ in the language of $X$ is at most $\lceil m\delta \rceil$ where $\lceil \cdot \rceil$ denotes the ceiling function. 

Now, further define a subshift $Y\subset\{0,\dots,k-1\}^{\N_0}$ obtained by changing some $1$s in a point of $X$ to any letters in $\{0, \ldots, k-1\}$. 
The number of words of length $m$ in $Y$ is at most $(m+1) k^{\lceil m\delta \rceil}$ for any $m\in\N$, so $h(Y) \leq \delta \log k$. 


Finally, for any function $f: S \to \{0, 1, \ldots, k-1\}$, we can extend $f$ to $x: \N_0 \to \{0, 1, \ldots, k-1\}$ by mapping to $0$ on $S^c$. The orbit closure of $f$ under the left shift is a subshift $X_f$ for which $x$ is a transitive point, and $X_f$ is contained in $Y$, so $h(X_f) \leq h(Y) \leq \delta \log k$. 
\end{proof}

\section{Uniquely ergodic and strictly ergodic systems with zero entropy}

\label{sec:strictly_ergodic}



We prove \cref{prop:uniquely_ergodic_one} in this section. First, we observe that if $d^*(S) > 0$, then $S$ is not a weak interpolation of all orders for uniquely ergodic systems.

\begin{Lemma}\label{thm:finite_alphabet_uniquely_ergodic}
If $d^*(S) > 0$, then there exists $k\in\N$ such that $S$ is not a weak interpolation set of order $k$ for the class of uniquely ergodic systems.
\end{Lemma}
\begin{proof}

Suppose $d^*(S) > 1/k$ for some $k \in \N$. 
We can find a sequence $(F_N)_{N \in \N}$ of pairwise disjoint intervals in $\N$ such that 
\[
    \lim_{N \to \infty} \frac{1}{|F_N|} \sum_{n \in F_N} 1_S(n) = d^*(S) > 1/k. 
\] 
Define the function $f: S \to \{0, 1, \ldots, k-1\}$ by $f(n) = i$ if $n \in F_N$ for some $N$ satisfying $N \equiv i \bmod k$ and arbitrary if $n \notin\bigcup_{N \in \N} F_N$.
We claim that there does not exist a uniquely ergodic system $(X, T)$, a continuous function $F: X \to \C$, and a transitive point $x \in X$ such that $f(n) = F(T^n x)$ for all $n \in S$. As a result, $S$ is not a weak interpolation set of order $k$ for the class of uniquely ergodic systems.

To prove the claim we proceed by contradiction.
Suppose that there exist a uniquely ergodic system $(X, T)$, a continuous function $F: X \to \C$, and a transitive point $x \in X$ such that $f(n) = F(T^n x)$ for all $n \in S$.
For each $0 \leq i < k$, define $X_i = F^{-1}(\{i\})$; then the sets $X_i$ are compact, nonempty, and disjoint subsets of $X$. 
We may then define nonnegative continuous functions $g_i: X \rightarrow [0,1]$, for $0 \leq i < k$, such that $g_i = 1$ on $X_i$ and $g_0(y)+\cdots+g_{k-1}(y)\leq1$ for all $y\in X$.

With a fixed $i$, any weak$^*$ limit point of the sequence $(\mu_m)_{m\in\N}$, where
$$\mu_m:=\frac1{|F_{mk+i}|}\sum_{n\in F_{mk+i}}\delta_{T^nx},$$
is a $T$-invariant measure. 
By unique ergodicity, we conclude that $\mu_m\to\mu$ as $m\to\infty$.
Since every $n\in S\cap F_{mk+i}$ satisfies $T^nx\in X_i$ and hence $g_i(T^nx)=1$, we conclude that $\int_Xg_i\d\mu=\lim_{m\to\infty}\int_Xg_i\d\mu_m>1/k$ for all $i$.
But then linearity of integration would imply that $\int_Xg_0+\cdots+g_{k-1}\d\mu>1$, contradicting the fact that $g_0(y)+\cdots+g_{k-1}(y)\leq1$ for all $y\in X$.

\end{proof}

The next proposition shows that if $d^*(S) = 0$, then $S$ is an interpolation set for strictly ergodic systems of zero entropy.
The main difficulty in this proof is the requirement that the system is minimal. 
Without minimality, given $d^*(S) = 0$, it is simple to prove that $S$ is an interpolation set for uniquely ergodic systems of zero entropy. Indeed, for any bounded function $f: S \to \C $, extend $f$ to $\tilde{f}: \N_0 \to \C$ by defining $\tilde{f} = 0$ on $\N_0 \setminus S$. It is shown in the proof of \cref{prop:interpolation_finite_entropy} that the orbit closure of $\tilde{f}$ under the left shift is uniquely ergodic with the only invariant measure being the Dirac measure at the constant sequence $0$ and so has zero entropy.

\begin{Proposition}\label{prop:strictly_ergodic_interpolation}
Every set of zero Banach density is an interpolation set for strictly ergodic systems of zero entropy.
\end{Proposition}

\begin{proof}
Let $S\subset\N$ be a set satisfying $d^*(S) = 0$.
Let $u: S \to \C$ be a bounded function and let $K\subset\C$ be a compact set containing $u(S)$ and $0$.
We construct $X\subset K^{\N_0}$ iteratively via auxiliary sequences $m_k \in \mathbb{N}$, compact sets $X_k \subset K^{m_k}$, and
$w_k \in X_k$. 
Define $m_0 = 1$, $X_0 = K$, and $w_0 = 0$. Now suppose that $m_k$, $X_k$, and $w_k$ are defined. Since $X_k$ is compact,
there exists a $2^{-k}$-spanning subset $T_k$ (in the $\ell^\infty$ metric). 
Since $d^*(S) = 0$, we can define
$m_{k+1}$ to be a multiple of $(2k+2)m_k$ larger than $(2k+2)m_k |T_k|$ with the property that for every interval $I\subset\N$ of length 
$m_{k+1}$, $\frac{|I \cap S|}{|I|} < ((2k+2)m_k)^{-1}$.

Define $X_{k+1}$ to be the set of all concatenations of sets of $\frac{m_{k+1}}{m_k}$ strings in $X_k$ with the following two properties: each $w \in T_k$ appears at least once, and $w_k$ appears at least $(1 - (k+1)^{-1}) \frac{m_{k+1}}{m_k}$ times. 
(This set is nonempty since $\frac{m_{k+1}}{m_k} \geq (k+1)|T_k|$.) Define $w_{k+1}$ to be any element in $X_{k+1}$ with $w_k$ as a prefix. Since $w_k$ is a prefix of $w_{k+1}$ for all $k$, we can define $y$ to be the limit of the $w_k$, and define $X$ to be the closure of the orbit of $y$. 

We first claim that $X$ has zero entropy. For this, consider any $k \in \N$ and $\epsilon > 0$, and we will bound $c(n_k, \epsilon, T)$ from above (see \cref{def:entropy}). Every subword $w$ of length $n_k$ of any $x \in X$ is a subword of a concatenation of two consecutive $A_k$-words, each of which is a concatenation of $A_{k-1}$-words in which all but a proportion of at most $1/k$ are $w_{k-1}$. Therefore, $w$ is determined by the location of the transition between the $A_{k-1}$-words, the locations of the elements not part of concatenated $w_{k-1}$, and those elements. Therefore, 
\[
c(n_k, \epsilon, T) \leq n_k (\lceil \epsilon^{-1} \rceil)^{2n_k/k} \sum_{i = 0}^{\lfloor 2n_k/k \rfloor} \binom{n_k}{i} 
\leq n_k^2 (\lceil \epsilon^{-1} \rceil)^{2n_k/k} \binom{n_k}{\lfloor 2n_k/k \rfloor}.
\]
By Stirling's approximation, $\log c(n_k, \epsilon, T)/n_k \rightarrow 0$, and since $\epsilon$ was arbitrary, $h(T) = 0$.

Now we prove that $y$ is uniformly recurrent. Indeed, for any neighborhood $U$ of $y$, there exists $k$ so that
$\{z \ : \ \rho(z_1 \ldots z_{m_k}, y_1 \ldots y_{m_k}) < 2^{-k}\} \subset U$. (Here $\rho$ is the metric on $X_k$.) 
By definition, $w_k = y_1 \ldots y_{m_k} \in X_k$, and so there exists $v \in T_k$ within distance $2^{-k}$ of $w_k$.
Finally, $y$ is a concatenation of elements in $X_{k+1}$, each of which contains $v$ by definition. So, the set of shifts of $y$ in $U$ is syndetic with gaps bounded by $2m_{k+1}$. 
Since $U$ was arbitrary, $y$ is uniformly recurrent, and so 
$X$ is minimal. 

Next, we claim that $X$ is uniquely ergodic. First, we note that every $x \in X$ is a shift of an infinite concatenation of $A_k$-words for each $k$, and that each $A_k$-word is a concatenation of $A_{k-1}$-words in which at most a proportion of $1/k$ are not $w_{k-1}$. Therefore, if $z \in \mathbb{C}$ does not occur in $w_{k-1}$, the frequency of occurrences of $z$ in $x$ is at most $1/k$. It follows that if $z$ does not occur in any $w_k$, then $z$ has density $0$ of occurrences in $x$. Since $x$ was arbitrary, this means by the ergodic theorem that $\mu([z]) = 0$ for all invariant measures $\mu$. (Here $[z]$ is the cylinder set $\{x \in X: x(0) = z\}$.)

Define $W$ to be the countable set of all complex numbers appearing in some $w_k$, and consider any cylinder of the form 
$R = \prod_{i=1}^m ([a_i, b_i] \times [c_i, d_i]) \times \prod_{m+1}^{\infty} K$, where all $a_i, b_i, c_i, d_i \notin W$.
Choose any ergodic measure $\mu$ and any $x \in X$ generic\footnote{In a system $(X, T)$, a point $x \in X$ is said to be \emph{generic} for an ergodic measure $\mu$ if $\frac{1}{N} \sum_{n=1}^N F(T^n x) \to \int_X F \ d \mu$ for every continuous function $F: X \to \C$.} for $\mu$. 
Note that since all endpoints are not in $W$, by earlier observation the boundary of $R$ has $\mu$-measure $0$, and so 
\[
    \mu(R) = \lim_{n \rightarrow \infty} \frac{1}{n} |\{0 \leq i < n \ : \ \sigma^i x \in R\}|.
\]

Consider now $k\in\N$ such that $m_k>m$. We note that $y$ is an infinite concatenation of elements in $X_{k+1}$, and since $X_{k+1}$ is closed, the same is true for $x$; therefore $x((im_{k+1}, (i+1)m_{k+1}]) \in X_{k+1}$ for all $k$. 
We denote 
\[
    N(w_k) = |\{0 < j \leq m_k - m \ : \ w_k(j) \ldots w_k(j+m-1) \in \prod_{i=1}^m ([a_i, b_i] \times [c_i, d_i])\}|.
\]
Then, since each word $x([im_{k+1}, (i+1)m_{k+1}))$ is a concatenation of elements of $X_k$ containing at least $(1 - (k+1)^{-1}) \frac{m_{k+1}}{m_k}$ occurrences of $w_k$, we see that
\begin{multline*}
iN(w_k) (1 - (k+1)^{-1}) \frac{m_{k+1}}{m_k} \leq 
|\{0 \leq j \leq im_{k+1} - m \ : \ \sigma^j x \in R\}|
\\
\leq 
iN(w_k) (1 - (k+1)^{-1}) \frac{m_{k+1}}{m_k} + i(k+1)^{-1} m_{k+1} + im_{k+1} \frac{m}{m_k}.
\end{multline*}
Here, the first term of the upper bound comes from visits to $R$ within one of the $w_k$ concatenated within some
$x((jm_{k+1}, (j+1)m_{k+1}]) \in X_{k+1}$, and the remaining terms are just an upper bound on the number of possible visits to $R$ 
at other locations.
By dividing by $L = im_{k+1}$ and taking limits as $i \rightarrow \infty$ (and recalling that $x$ is generic for $\mu$), we get
\[
(1 - (k+1)^{-1}) \frac{N(w_k)}{m_k} \leq \mu(R) \leq (1 - (k+1)^{-1}) \frac{N(w_k)}{m_k} + (k+1)^{-1} + \frac{m}{m_{k}}.
\]

This estimate holds for all $k$, and so we can let $k \rightarrow \infty$ to see that $\mu(R) = \lim_k \frac{N(w_k)}{m_k}$ (and particular this limit exists). 
Since this quantity depends only on the sequence $w_k$, $\mu(R)$ does not depend on $x$. Since the collection of cylinders $R$ generates the Borel $\sigma$-algebra, $\mu$ does not depend on $x$, and so $X$ is uniquely ergodic.

It remains to construct $x_u \in X$ for which $x(s) = u(s)$ for all $s \in S$. The construction of 
$x_u$ proceeds in steps, where it is continually assigned values in $K$ on more and more of $\mathbb{N}$, and undefined portions are labeled by $*$. Formally, define $x^{(0)} \in (K \cup \{*\})^{\mathbb{N}_0}$ by $x^{(0)}(s) = u(s)$ for $s \in S$ and $*$ for all other locations. 
Note (for induction purposes) that $x^{(0)}$ is an infinite concatenation of elements of $X_0$ and blocks of $*$ of length $m_0 = 1$, and that $x^{(0)}$ consists of 
$*$ on any ``interval'' $(im_0, (i+1)m_0]$ which is disjoint from $S$.


Now, suppose that $x^{(k)}$ has been defined as an infinite concatenation of elements of $X_k$ and blocks of $*$ of length $m_k$ which consists of only $*$ on any interval $(im_k, (i+1)m_k]$ which is disjoint from $S$. We wish to extend $x^{(k)}$ to $x^{(k+1)}$ by changing some 
$*$ symbols to values in $K$. 
Consider any
$i$ for which $S \cap (im_{k+1}, \ldots, (i+1)m_{k+1}] \neq \varnothing$. 
The portion of $x^{(k)}$ occupying that interval is a concatenation of elements of $X_k$ and blocks of $*$ of length $m_k$. 
By definition of $m_{k+1}$, $|S \cap (im_{k+1}, \ldots, (i+1)m_{k+1}]| < \frac{m_{k+1}}{(2k+2)m_k}$. 
This means that if 
$x^{(k)}((im_{k+1}, \ldots, (i+1)m_{k+1}])$ is written as a concatenation of elements of $X_k$ and blocks of $*$ of length $m_k$,
at least $(1 - (2k+2)^{-1}) \frac{m_{k+1}}{m_k}$ are blocks of $*$. 
Overwrite any $(1 - (k+1)^{-1}) \frac{m_{k+1}}{m_k}$ of these with $w_k$, leaving at least $\frac{m_{k+1}}{(2k+2)m_k} \geq |T_k|$. Fill all remaining ones with elements of $T_k$ in a way that each is used at least once. 

By definition, this creates a word in $X_{k+1}$, which we denote by $w^{(k+1)}_i$. Define $x^{(k+1)}((im_{k+1},$ $ (i+1)m_{k+1}]) = w^{(k+1)}_i$ for any $i$ as above (i.e. those for which $S \cap (im_{k+1}, \ldots, (i+1)m_{k+1}] \neq \varnothing$) and as $*$ elsewhere. Note that $x^{(k+1)}$ is an infinite concatenation of elements of $X_{k+1}$ and blocks of $*$ of length $m_{k+1}$ which contains $*$ on any interval $(im_{k+1}, (i+1)m_{k+1}]$ which is disjoint from $S$, and that $x^{(k+1)}$ agrees with $x^{(k)}$ on all locations where $x^{(k)}$ did not contain $*$.

Continue in this way to build a sequence $x^{(k)}$. Since each is obtained from the previous by changing some $*$s to values in $K$ (which are not changed in future steps), they approach a limit $x_u$. Since $S \neq \varnothing$, $S \cap (0, m_k] \neq \varnothing$ for all large enough $k$, and so 
$x^{(k)}((0, m_k])$ has no $*$, meaning that $x_u \in K^{\mathbb{N}_0}$. By definition, $x_u$ agrees with $x^{(0)}$ on all locations where $x^{(0)}$ did not have $*$, and so $x_u(s) = u(s)$ for all $s \in S$. It remains only to show that $x_u \in X$. To see this, we note that for all sufficiently large $k$, $x_u(0, m_k] \in X_k$ by definition. Also by definition, for all sufficiently large $k$, 
$y$ begins with $w_{k+1}$, which contains a subword in $T_k$ which is within distance $2^{-k}$ of $x_u(0, m_k]$. This yields a sequence of shifts of $y$ which converges to $x_u$, proving that $x_u \in X$ and $x_u(s) = u(s)$ for all $s \in S$. 
Finally, this implies that $S$ is interpolation for strictly ergodic systems by considering $F \in C(X)$ defined by $F(x) = x(0)$. 
\end{proof}

We are ready to prove \cref{prop:uniquely_ergodic_one} whose  statement is repeated here for convenience.
\begin{named}{\cref{prop:uniquely_ergodic_one}}{}
    Let $S \subset \N$. The following are equivalent:
\begin{enumerate}
    \item \label{item:weak-unique-ergodic} $S$ is a weak interpolation set of all orders for uniquely ergodic systems.
    
    \item \label{item:strong-unique-ergodic} $S$ is an interpolation set for uniquely ergodic systems.

    \item \label{item:strong-strict-ergodic} $S$ is an interpolation set for strictly ergodic systems.

    \item \label{item:strong-strict-ergodic-zero-entropy} $S$ is an interpolation set for strictly ergodic systems of zero entropy.

    \item \label{item:unique-ergodic-zero-banach-density} $d^*(S) = 0$.
\end{enumerate}
\end{named}
\begin{proof}
    It is obvious that $(\ref{item:strong-strict-ergodic-zero-entropy}) \Rightarrow (\ref{item:strong-strict-ergodic}) \Rightarrow (\ref{item:strong-unique-ergodic}) \Rightarrow (\ref{item:weak-unique-ergodic})$. The implication $(\ref{item:unique-ergodic-zero-banach-density}) \Rightarrow (\ref{item:strong-strict-ergodic-zero-entropy})$ follows from \cref{prop:strictly_ergodic_interpolation} and $(\ref{item:weak-unique-ergodic}) \Rightarrow (\ref{item:unique-ergodic-zero-banach-density})$ follows from \cref{thm:finite_alphabet_uniquely_ergodic}.
\end{proof}

Lastly, we prove \cref{prop:weak_interpolation_uniquely_ergodic}; we recall the statement here for convenience.

\begin{named}{\cref{prop:weak_interpolation_uniquely_ergodic}}{}
    If $S \subset \mathbb{N}$ is syndetic and the word complexity of the
    sequence $1_S$ grows subexponentially (i.e. the orbit closure of $1_S$ under the left shift has zero entropy), then $S$ is not a weak interpolation set of order $2$ for uniquely ergodic systems.
\end{named}

\begin{proof}

Suppose for a contradiction that a set $S$ exists which is a weak interpolation set of order $2$ for uniquely ergodic systems and which is syndetic and has subexponential word complexity.
Say that $S$ is syndetic with gaps less than $G$, and then for every interval $I$ of length at least $2G$, 
$|S \cap I| > (2G)^{-1} |I|$. 

For any $n > 2G$, denote by $p(n)$ the word complexity of $\chi_S$, i.e. the number of different sets which occur as
$(S \cap [k, k+n)) - k$ for $k \in \mathbb{N}$. 
By subadditivity of upper density, there exists such a set $J = \{j_1, \ldots, j_{|J|}\}$ which occurs with upper density at least $(p(n))^{-1}$, i.e. $A = \{k \ : \ (S \cap [k, k+n)) - k = J\}$ has upper density at least $(p(n))^{-1}$. 
Note that $|J| > (2G)^{-1} n$.
By partitioning $A$ by residue classes modulo $n$, we can pass to a subset $B$ of $A$ with gaps of at least $n$ and upper density at least $(np(n))^{-1}$. 
It is easy to partition $B$ into subsets $B_i$, $1 \leq i \leq 2^{|J|}$, each with upper density $(np(n))^{-1}$. 

Enumerate the words in $\{0,1\}^{|J|}$ as $w_1, \ldots, w_{2^{|J|}}$. Define $y: S \rightarrow \{0,1\}$ by assigning, for every $1 \leq i \leq 2^{|J|}$, every $m \in B_i$, and every $1 \leq k \leq |J|$, $y(m + j_k) = w_i(k)$, and defining $y$ arbitrarily otherwise. This is well-defined since we assumed $B$ to have gaps of at least $n \geq |J|$. By assumption, there exists a uniquely ergodic system $(X, T)$, transitive point $x$, and $f \in C(X)$ so that $f(T^x x) = y(s)$ for all $s \in S$. 

Denote the unique measure of $X$ by $\mu$. For any $1 \leq i \leq 2^{|J|}$, any $m \in B_i$, and any $1 \leq k \leq |J|$,
$f(\sigma^{m+j_k} x) = y(m + j_k) = w_i(k)$. Therefore, for each such $m$, 
$\sigma^m(x) \in \bigcap_{k = 1}^{|J|} \sigma^{-j_k} f^{-1}\{w_i(k)\}$; denote the latter set by $C_i$. 
By unique ergodicity and the fact that $B_i$ has upper density at least $(np(n))^{-1}$, 
$\mu(C_i) \geq (np(n))^{-1}$. However, the sets $C_i$ are disjoint, so
$2^{|J|} (np(n))^{-1} \leq 1$. Since $|J| > (4G)^{-1} n$, we contradict subexponential growth of $p(n)$ for large enough $n$. Our original hypothesis was false, and $S$ is not a weak interpolation set of order $2$ for uniquely ergodic systems. 
\end{proof}

\section{Interpolation sets for distal systems and connections to pointwise recurrence}
\label{sec:distal}

A topological system $(X, T)$ on a metric space $X$ with metric $\rho$ is called \emph{distal} if for all distinct $x, y \in X$ one has $\inf_{n \in \N} \rho(T^n x, T^n y) > 0$.  The class of distal systems is closed under taking products and inverse limits. It is a subclass of the class of zero entropy systems and it contains all rotations on compact abelian groups and, more generally, all nilsystems \cite{Furstenberg-structure-distal-flows, Host_Kra_18}. Distal systems have been studied in ergodic theory for a long time, partly because their measure theoretic counterparts are core components in the structure theory of measure preserving systems (see Furstenberg \cite{Furstenberg77} and Zimmer \cite{Zimmer76, Zimmer76a}).

Regarding interpolation sets for distal systems, it is shown in \cite[Corollary 5.1]{Pavlov-2008} that any set of positive integers $S=\{s_1<s_2<\ldots\}$ satisfying 
\begin{equation}
\label{eqn_rc_1}
\limsup_{n\in\N}\frac{\log(s_{n+1})}{\log(s_{n+1}-s_n)}<\infty
\end{equation}
is an interpolation set for distal systems. (In fact, in this case, $S$ is shown to be an interpolation set for a special subclass of distal systems called skew products on tori.)
The following natural question asked in \cite{Le_sublac} is still open.



\begin{Question}\label{ques:distal_union_finite_set}
Let $S \subset \N$ be an interpolation set for distal systems and $F \subset \N$ finite. Is it true that $S \cup F$ is an interpolation set for distal systems?
\end{Question}

Note that \cref{ques:distal_union_finite_set} has an affirmative answer if the class of distal systems is replaced with rotations on compact abelian groups (cf.~\cite{Le_interpolation_nil}), nilsystems (cf.~\cite{Le_sublac}), totally transitive systems (cf.~\cref{thm:total_transitive_weak_mixing_mixing}), weak mixing systems (cf.~\cref{thm:total_transitive_weak_mixing_mixing}), strong mixing systems (cf.~\cref{thm:total_transitive_weak_mixing_mixing}), minimal systems (cf.~\cref{thm:interpolation_minimal_systems_new}), or uniquely ergodic systems (cf.~\cref{prop:uniquely_ergodic_one}). This prompts us to believe that the answer for distal systems is also positive. 


There exists a noteworthy connection between \cref{ques:distal_union_finite_set} and the notion of pointwise recurrence in distal systems. A set $R\subset\N$ is a \emph{set of pointwise recurrence for distal systems} if for every distal system $(X,T)$ and every point $x\in X$, we have $\inf_{n\in R} \rho(x, T^nx)=0$.



\begin{Proposition}
\label{lem_connection_rec_interpol_distal}
    If every set of pointwise recurrence for distal systems can be partitioned into two disjoint sets of pointwise recurrence for distal systems, then the answer to \cref{ques:distal_union_finite_set} is positive.
\end{Proposition}

\begin{proof}
Let $S$ be an interpolation set for distal systems and let $a \in \N \setminus S$. First, we show that $S - a$ is not a set of pointwise recurrence for distal systems. Indeed, for contradiction, suppose the opposite. 
Then by our hypothesis, $S - a$ can be partitioned into two sets $A$ and $B$ both of which are sets of pointwise recurrence for distal systems. Therefore, for any distal system $(X, T)$, any point $x \in X$, 
\[
    x \in \overline{\{T^n x: n \in A\}} \cap \overline{\{T^n x: n \in B\}}.
\]
(Here we use the fact that distal systems are semisimple, i.e. the orbit closure of every point is minimal.) It follows that
\begin{equation}\label{eq:A+a_B+a}
    T^a x \in \overline{\{T^n x: n \in A + a\}} \cap \overline{\{T^n x: n \in B + a\}}.
\end{equation}
Note that $A + a$ and $B + a$ are disjoint subsets of $S$ and so we can define a function $f: S \to \{0, 1\}$ with $f|_{A+a} = 0$ and $f|_{B+a} = 1$. Because the intersection in \eqref{eq:A+a_B+a} is nonempty for any distal system $(X, T)$ and any point $x \in X$, the function $f$ cannot be extended to a sequence coming from a distal system. This contradicts our assumption that $S$ is an interpolation set for distal systems.


To show $S \cup \{a\}$ is an interpolation set for distal systems, let $g: S \cup \{a\} \to [0, 1]$ be an arbitrary function. 
Since $S$ is an interpolation set for distal systems, there exists a distal system $(X, T)$, a point $x_0 \in X$ and a function $F \in C(X)$ such that $F(T^n x_0) = g(n)$ for all $n \in S$. Because $S - a$ is not a set of pointwise recurrence for distal systems, there exists a distal system $(Y, S)$ and a point $y_0 \in Y$ such that $y_0 \not \in \overline{\{S^n y_0: n \in S - a\}}$, and so $S^a y_0 \not \in \overline{\{S^n y_0: n \in S\}}$. Considering the product system $(X \times Y, T \times S)$ and the point $(x_0, y_0)$, we have
\[
    (T \times S)^a (x_0, y_0) \not \in  \overline{\{(T \times S)^n (x_0, y_0): n \in S\}}.
\]
Therefore, there exists a continuous function $G: X \times Y \to [0, 1]$ such that 
\[
    G(x, y) = \begin{cases}
        F(x), \text{ if } (x, y) \in \overline{\{(T \times S)^n (x_0, y_0): n \in S\}}, \\
        g(a), \text{ if } (x, y) = (T \times S)^a (x_0, y_0).
    \end{cases}
\]
It follows that 
\[
    G((T \times S)^n (x_0, y_0)) = \begin{cases}
        F(T^n x_0) = g(n), \text{ if }  n \in S, \\
        g(a), \text{ if } n = a.
    \end{cases}
\]
Since $g: S \cup \{a\} \to [0, 1]$ is arbitrary, we conclude that $S \cup \{a\}$ is an interpolation set for distal systems. 
\end{proof}


In light of \cref{lem_connection_rec_interpol_distal}, the following open question arises naturally and warrants further investigation.

\begin{Question}\label{ques:distal_point_recurrence}
    Is it true that every set of pointwise recurrence for distal systems can be partitioned into two disjoint sets of pointwise recurrence for distal systems?
\end{Question}

We remark that the analogues of \cref{ques:distal_point_recurrence} for compact abelian group rotations and nilsystems were answered in \cite{Le_interpolation_nil, Le_sublac, Ryll-Nardzewski_1964}.

A set $E\subset\N$ is an \emph{IP-set} if it contains an infinite sequence and all its finite sums, that is, there are $n_1<n_2<\ldots\in\N$ with $\{n_{i_1}+\ldots+n_{i_k}: k\in\N,~i_1<\ldots<i_k\}\subset E$.  
It is well known that IP-sets are sets of pointwise recurrence in distal systems 
\cite[Theorem 9.11]{Furstenberg81} and it is not difficult to see that every IP-set contains two disjoint subsets which are both IP-sets.
Therefore, an affirmative answer to the following question, asked by Host, Kra, and Maass in~\cite{Host-Kra-Maass-2016}, would imply an affirmative answer to \cref{ques:distal_point_recurrence}.  


\begin{Question}[Host-Kra-Maass {\cite[Question 3.11]{Host-Kra-Maass-2016}}]\label{ques:host_kra_maass}
    Is it true that every set of pointwise recurrence for distal systems is an $IP$-set?
\end{Question}

While we don't know how to answer \cref{ques:distal_point_recurrence}, our next result shows that the stronger \cref{ques:host_kra_maass} has a negative answer.


\begin{Theorem}
\label{thm:pointwise_distal_no_IP}
    There exists a set $F \subset \N$ that is a set of pointwise recurrence for distal systems but such that for any $x,y\in F$ we have $x+y\notin F$. 
    In particular, $F$ is not an $IP$-set.
\end{Theorem}


\begin{proof}
First we show that if $F \subset \N$ and $F - n$ contains an $IP$-set for all $n \in \N$, then $F$ is a set of pointwise recurrence for distal systems. 
Indeed, let $(X, T)$ be a distal system, $U\subset X$ be open and $x \in U$. 
Since the orbit closure of every point in a distal system is minimal, there is $n \in \N$ such that $T^n x \in U$. 
Since $F - n$ contains an $IP$-set, by Furstenberg's theorem \cite[Theorem 9.11]{Furstenberg81}, there exists $m \in F - n$ such that $T^m (T^n x) \in U$. 
It follows that $T^{m+n} x \in U$ and note that $m + n \in F$.   

Now we will construct the set $F$. 
Let $I_1, I_2, I_3, \ldots$ be pairwise disjoint infinite subsets of $\N$ such that $\min \{k: k \in I_n\} \geq n$ for all $n$. 
Let $J_n$ be the $IP$-set generated by $\{10^k: k \in I_n\}$. 
Then the sets $J_1,J_2,J_3,\ldots$ are pairwise disjoint and every element of $J_n$ is divisible by $10^n$. 

Let $F = \bigcup_{n \in \N} (J_n + n)$. 
Then for every $n \in \N$, the shift $F - n$ contains the $IP$-set $J_n$ and hence $F$ is a set of pointwise recurrence for distal systems. It remains to show that $F$ does not contain a triple of the form $x, y, x + y$. By contradiction, assume there are $n, m, k$ (not necessarily distinct) such that 
\begin{equation}
\label{eqn_decimal_sum_0}
    \big((J_n + n) + (J_m + m)\big) \cap (J_k + k)  \neq \varnothing.
\end{equation}
In other words, there are $1 \leq n_1 < n_2 < \ldots < n_s$, $1 \leq m_1 < m_2 < \ldots < m_t$ and $1 \leq k_1 < k_2 < \ldots < k_u$ such that
\begin{equation}
\label{eqn_decimal_sum_1}
    (10^{n_1} + \ldots + 10^{n_s} + n) + (10^{m_1} + \ldots + 10^{m_t} + m) = 10^{k_1} + \ldots + 10^{k_u} + k.
\end{equation}
By removing the terms that appear in both sides, we can assume that $\{k_1,\dots,k_u\}$ is disjoint from $\{n_1,\dots,n_s,m_1,\dots,m_t\}$.
In particular, $k_u$ must be strictly larger than both $n_s$ and $m_t$, for if, say, $n_s$ was the larger, then 
because $k \leq k_1 < 10^{k_1}$, we would have
\begin{equation*}
    10^{n_s} > 10^{k_1} + \ldots + 2\cdot10^{k_u}> 10^{k_1} + \ldots + 10^{k_u} + k,
\end{equation*}
which contradicts \eqref{eqn_decimal_sum_1}. 
That leaves the only possibility that $k_u$ is the largest. 
But this is also impossible since
\begin{equation*}
    10^{k_u} > 2(1 + 10^1 + \ldots + 10^{k_u - 1}) > (10^{n_1} + \ldots + 10^{n_s} + n) + (10^{m_1} + \ldots + 10^{m_t} + m).
\end{equation*}
Thus, the intersection in \eqref{eqn_decimal_sum_0} is empty and hence $F$ does not contain a triple of the form $x,y,x+y$. 
\end{proof}





\section{Open questions}\label{sec_Qs}



In this section, we list some natural questions that arise from our study and may be of interest to readers. 











The first question is concerned with interpolation sets for compact abelian group rotations. A \emph{compact abelian group rotation} (or \emph{group rotation} for short) has the form $(G, g)$ where $G$ is a compact abelian group and $g \in G: x \mapsto g \cdot x$ for $x \in G$. It was shown in \cite{Strzelecki_1963} that lacunary sets are interpolation sets for group rotations. On the other hand, no subexponential sets are interpolation sets for this class of systems \cite{Le_interpolation_nil}. 
For the sets $S$ that are neither lacunary nor subexponential, whether $S$ is an interpolation set for group rotations hinges on some delicate arithmetic properties instead of mere density. For example, $\{2^n: n \in \N\} \cup \{2^n + 2n: n \in \N\}$ is not interpolation while $\{2^n: n \in \N\} \cup \{2^n + 2n + 1: n \in \N\}$ is. Based on these facts, it is desirable to have a complete combinatorial characterization of interpolation sets for group rotations:


\begin{Problem}\label{ques:group_rotations}
    Give a combinatorial characterization of interpolation sets for compact abelian group rotations. 
\end{Problem}

Group rotations is a subclass of inverse limits of nilsystems, and so our next two questions consider interpolation sets for the latter, more general family of systems. 
For $k\in \N$, \emph{$k$-step nilsystem} is a topological dynamical system $(X,T)$ where $X=G/\Gamma$ for $G$ a $k$-step nilpotent Lie group and $\Gamma$ a discrete co-compact subgroup of $G$, and $T\colon X\to X$ is left translations by a fixed group element $g \in G$.
Because a group rotation is an inverse limit of $1$-step nilsytems, an interpolation set for group rotations is interpolation for inverse limits of $k$-step nilsystems for any $k \geq 1$. 
However, it was shown in \cite{Le_interpolation_nil} that there is an interpolation set for inverse limits of $2$-step nilsystems which is not interpolation for group rotations. It is unknown if the same feature continues to hold for higher step nilsystems:

\begin{Question}\label{ques:separate_steps_nil}
    For $k \geq 2$, does there exist an interpolation set for inverse limits of $(k+1)$-step nilsystems that is not an interpolation set for inverse limits of $k$-step nilsystems?
\end{Question}
We remark that \cref{ques:separate_steps_nil} was also asked in \cite{Le_interpolation_nil, Le_sublac}.

Similarly to the situation with group rotations, we have a relatively good understanding of interpolation sets for inverse limits of nilsystems in terms of density \cite{Briet_Green_2021, Le_sublac, Strzelecki_1963}: if $S$ is lacunary, $S$ is an interpolation for inverse limits of nilsystems; if $S$ is subexponential, $S$ is not interpolation for this class. However, this characterization does not cover all subsets of $\N$ and so a complete characterization of interpolation sets of inverse limits of nilsystems is still missing:



\begin{Problem}
For $k \geq 2$, give a characterization for interpolation sets of inverse limits of $k$-step nilsystems.
\end{Problem}


Next, we shift our focus to interpolation sets for distal systems. 
The orbit closure of every point in a distal system is minimal and therefore, by \cref{thm:interpolation_minimal_systems_new}, an interpolation set for distal systems must be non-piecewise syndetic. 

Moreover it follows from \eqref{eqn_rc_1} that any
sequence $s_n$ for which the difference between consecutive elements $s_{n+1}-s_n$ grows sufficiently fast (e.g. \{$\lfloor n^{1+\epsilon}\rfloor:n\in\N\}$ for any $\epsilon>0$), yields an interpolation set for distal systems. 
This observation naturally leads to the following question, aiming to weaken the condition \eqref{eqn_rc_1}:


\begin{Question}\label{ques:interplation_for_distal}
    Is it true that any set $S = \{n_1 < n_2 < \ldots\}$ for which the gaps $n_{k+1} - n_k$ tend to infinity is an interpolation set for distal systems?
\end{Question}

An important aspect, both in the condition \eqref{eqn_rc_1} and in \cref{ques:interplation_for_distal}, is the emphasis on the growth of gaps $n_{k+1} - n_k$ rather than the growth of the sequence $n_k$ itself. Controlling the growth rate of $n_k$ is not sufficient for interpolation in distal systems. Indeed, one can construct examples of sets of arbitrarily fast growth rate that fail to be interpolation sets for distal systems by considering sparse $IP$-sets. Yet $IP$-sets do not  provide a counterexample to \cref{ques:interplation_for_distal} since they possess gaps that appear infinitely often.

We also have the following, more concrete, question about distal systems involving prime numbers.



\begin{Question}\label{ques:prime_distal}
    Is it true that the set of primes $\P$ is an interpolation set for distal systems?
\end{Question}
According to the celebrated Zhang's theorem \cite{Zhang_boundedgaps}, the sequence of gaps between consecutive primes does not go to infinity. Therefore, \cref{ques:prime_distal} is not a special case of \cref{ques:interplation_for_distal}. In order for $\P$ to be an interpolation set for distal systems, for any partition $\P = A \cup B$, there must be a distal system $(X, T)$ and a point $x \in X$ such that the closures $\overline{\{T^p x: p \in A\}}$ and $\overline{\{T^p x: p \in B\}}$ are disjoint. Related to this,
recently, Kanigowski, \Lemanczyk{}, and Radziwi{\l\l} \cite{Lemanczyk-Kanigowski-Radziwill-primenumbertheoremskewproduct} proved that the prime number theorem holds for skew products on $\T^2$ with $T(x, y) = (x + \alpha, y + h(x))$ where $h$ is an analytic function. More precisely, they showed that for these systems $(X, T)$ and for any $x \in X$, the set $\{T^p x: p \in \P\}$ is uniformly distributed on $\T^2$. That means for any partition $\P = A \cup B$, 
\[
    \overline{\{T^p x: p \in A\}} \cup \overline{\{T^p: p \in B\}} = \T^2
\]
and so these two closures cannot be disjoint. As a result, $\P$ is not interpolation for the skew products on $\T^2$ with analytic skewing functions. That being said, since the class of distal systems is much larger than the aforementioned skew products, the answer to \cref{ques:prime_distal} may be positive.




The next question is concerned with the threshold for density of weak interpolation sets for uniquely ergodic systems. It follows from the proof of \cite[Theorem 8.1]{Weiss-SingleOrbitDynamics} that for any $\epsilon > 0$, there exists a set $S$ which is a weak interpolation set of order $2$ for uniquely ergodic systems such that $d^*(S) > 1/2 - \epsilon$. The proof can be generalized to arbitrary $k \in \N$ to show that there exists a set $S$ which is a weak interpolation set of order $k$ for uniquely ergodic systems such that $d^*(S) > 1/k - \epsilon$.
However, from the proof of \cref{prop:uniquely_ergodic_one}, there is no such a set that satisfies $d^*(S) > 1/k$. 
These facts lead to the following question. 

\begin{Question}
    For $k \in \N$, is there a set $S$ which is a weak interpolation set of order $k$ for uniquely ergodic systems such that $d^*(S) = 1/k$?
\end{Question}
Among the classes of systems considered in this paper, for minimal systems, totally minimal systems, systems of zero entropy, and systems of finite entropy, weak interpolation of order $2$ is equivalent to weak interpolation of all orders. On the other hand, for uniquely ergodic systems and strictly ergodic systems, these two notions are not equivalent. 
Our next question explores this equivalence for the class of totally transitive systems. More precisely, \cref{thm:total_transitive_weak_mixing_mixing} implies that a weak interpolation set of all orders for totally transitive systems cannot be syndetic. Nevertheless, we do not know whether a weak interpolation set of order $2$ for this class can be syndetic or not.
\begin{Question}\label{Q2synd}
Does there exist a syndetic set which is weak interpolation of order $2$ for the class of totally transitive systems?
\end{Question}

Our last question is concerned with the necessity of the hypothesis of \cref{prop:weak_interpolation_uniquely_ergodic}. Recall that \cref{prop:weak_interpolation_uniquely_ergodic} shows that a syndetic set $S$ satisfying the orbit closure of $1_S$ under the left shift has zero entropy cannot be weak interpolation of order $2$ for uniquely ergodic systems. We do not know whether the entropy condition is necessary and so the following question is open:

\begin{Question}\label{Q2synd-2}
Does there exist a syndetic set which is weak interpolation of order $2$ for the class of uniquely ergodic systems?
\end{Question}





\bibliographystyle{plain}
\bibliography{refs-interpolation}


\bigskip
\footnotesize
\noindent
Andreas Koutsogiannis \\
\textsc{Aristotle University of Thessaloniki}\par\nopagebreak
\noindent
\href{mailto:akoutsogiannis@math.auth.gr}
{\texttt{akoutsogiannis@math.auth.gr}}

\bigskip
\footnotesize
\noindent
Anh N.\ Le\\
\textsc{University of Denver} \par\nopagebreak
\noindent
\href{mailto:anh.n.le@du.edu}
{\texttt{anh.n.le@du.edu}}

\bigskip
\footnotesize
\noindent
Joel Moreira\\
\textsc{University of Warwick} \par\nopagebreak
\noindent
\href{mailto:joel.moreira@warwick.ac.uk}
{\texttt{joel.moreira@warwick.ac.uk}}

\bigskip
\footnotesize
\noindent
Ronnie Pavlov\\
\textsc{University of Denver} \par\nopagebreak
\noindent
\href{mailto:rpavlov@du.edu}
{\texttt{rpavlov@du.edu}}

\bigskip
\footnotesize
\noindent
Florian K.\ Richter\\
\textsc{{\'E}cole Polytechnique F{\'e}d{\'e}rale de Lausanne} (EPFL)\par\nopagebreak
\noindent
\href{mailto:f.richter@epfl.ch}
{\texttt{f.richter@epfl.ch}}

\end{document}